\documentclass[11pt, a4paper]{article}
\usepackage{times}
\usepackage{a4wide}
\usepackage{hyperref}
\usepackage[british]{babel}
\usepackage{enumerate, longtable}
\usepackage{amsmath, amscd, amsfonts, amsthm, amssymb, latexsym, comment, stmaryrd, graphicx, color}
\usepackage[all]{xypic}
\usepackage[T1]{fontenc}
\usepackage[utf8]{inputenc}

\newtheorem{thm}{Theorem}[section]
\newtheorem{lem}[thm]{Lemma}
\newtheorem{defi}[thm]{Definition}
\newtheorem{rem}[thm]{Remark}

\newtheorem{prop}[thm]{Proposition}
\newtheorem{cor}[thm]{Corollary}
\newtheorem{notation}[thm]{Notation}
\newtheorem{setup}[thm]{Set-up}

\newcommand{\Hom}{{\rm Hom}}
\newcommand{\GL}{\mathrm{GL}}
\newcommand{\SL}{\mathrm{SL}}
\newcommand{\univ}{\mathrm{univ}}
\newcommand{\fin}{\mathrm{fin}}
\newcommand{\real}{\mathrm{real}}
\newcommand{\minim}{\mathrm{min}}
\newcommand{\id}{\mathrm{id}}
\newcommand{\ad}{\mathrm{ad}}
\newcommand{\dih}{\mathrm{dih}}
\newcommand{\incl}{\mathrm{incl}}
\newcommand{\Ind}{{\rm Ind}}

\newcommand{\ab}{\mathrm{ab}}
\newcommand{\inft}{\mathrm{inf}}
\newcommand{\tr}{\mathrm{tr}}

\DeclareMathOperator{\Image}{im}
\DeclareMathOperator{\Gal}{Gal}
\DeclareMathOperator{\Aut}{Aut}
\DeclareMathOperator{\Frob}{Frob}
\DeclareMathOperator{\Tr}{Tr}
\DeclareMathOperator{\h}{H}
\DeclareMathOperator{\hhat}{\hat{H}}

\newcommand{\plim}[1]{\underset{#1}{\underleftarrow{\lim} \;}}

\newcommand{\mat}[4]{
 \left(  \begin{smallmatrix} #1 & #2 \\ #3 & #4 \end{smallmatrix} \right)}
\newcommand{\vect}[2]{
 \left(  \begin{smallmatrix} #1 \\ #2 \end{smallmatrix} \right)}

\newcommand{\calC}{\mathcal{C}}

\newcommand{\calG}{\mathcal{G}}

\newcommand{\calM}{\mathcal{M}}

\newcommand{\calO}{\mathcal{O}}

\newcommand{\calU}{\mathcal{U}}

\newcommand{\fm}{\mathfrak{m}}
\newcommand{\fn}{\mathfrak{n}}

\newcommand{\fq}{\mathfrak{q}}

\newcommand{\CC}{\mathbb{C}}

\newcommand{\FF}{\mathbb{F}}

\newcommand{\NN}{\mathbb{N}}

\newcommand{\QQ}{\mathbb{Q}}

\newcommand{\TT}{\mathbb{T}}

\newcommand{\ZZ}{\mathbb{Z}}

\newcommand{\Qbar}{\overline{\QQ}}

\newcommand{\Kbar}{\overline{K}}
\newcommand{\Hbar}{{\overline{H}}}
\newcommand{\Gbar}{{\overline{G}}}
\newcommand{\chibar}{{\overline{\chi}}}
\newcommand{\chibarhat}{\widehat{\overline{\chi}}}

\newcommand{\rhobar}{{\overline{\rho}}}

\begin{document}

\title{Dihedral Universal Deformations}
\author{Shaunak V.\ Deo \and Gabor Wiese}


\maketitle

\begin{abstract}
This article deals with universal deformations of dihedral representations with a particular focus on the question
when the universal deformation is dihedral. Results are obtained in three settings:
(1) representation theory, (2) algebraic number theory, (3) modularity.
As to (1), we prove that the universal deformation is dihedral if all infinitesimal deformations are dihedral.
Concerning (2) in the setting of Galois representations of number fields, we give sufficient conditions to ensure that the universal deformation
relatively unramified outside a finite set of primes is dihedral, and discuss in how far these conditions are necessary.
As side-results, we obtain cases of the unramified Fontaine--Mazur conjecture, and in many cases positively answer a question of Greenberg and Coleman on the
splitting behaviour at~$p$ of $p$-adic Galois representations attached to newforms.
As to (3), we prove a modularity theorem of the form `$R=\TT$' for parallel weight one Hilbert modular forms for cases when the
minimal universal deformation is dihedral.

MS Classification: 11F80 (primary), 11F41, 11R29, 11R37
\end{abstract}

\section{Introduction}

The basic object in this article is a continuous absolutely irreducible representation
$$ \rhobar: G \to \GL_2(\FF)$$
that is  {\em dihedral} in the sense that it is induced from a character,
where $G$ is a profinite group and $\FF$ is a finite field of characteristic~$p$.
We consider a deformation $\rho: G \to \GL_2(R)$ of $\rhobar$ for any complete local Noetherian algebra $R$ over $W(\FF)$, the ring of Witt vectors of~$\FF$, with residue field~$\FF$.
We prove results in the following three settings:

\begin{enumerate}[(1)]
\item \underline{Representation theory results:}

We prove useful alternate characterisations to the dihedral property of a deformation $\rho$ of $\rhobar$ as above.
We also prove that being dihedral is an {\em infinitesimal property}, in the following sense:
the universal deformation of $\rhobar$ is dihedral if and only if all infinitesimal deformations are dihedral.

\item \underline{Number theory results:}

Here we let $G$ be $G_K = \Gal(\Kbar/K)$, the absolute Galois group of a number field~$K$. We give sufficient conditions, using class field theory,
to ensure that the universal deformation of~$\rhobar$ relatively unramified outside a finite set of primes remains dihedral.
In those cases, we compute the structure of the corresponding universal deformation ring and discuss in a series of remarks in how far the sufficient conditions are necessary.
We apply our results on the one hand to Boston's strengthening of the unramified Fontaine--Mazur conjecture.
On the other hand, we positively answer many cases of a question of Greenberg and Coleman on the splitting behaviour at~$p$
of the $p$-adic Galois representation attached to a newform.

\item \underline{Modularity results (an `$R=\TT$-theorem'):}

Assume in addition that the number field~$K$ is totally real, that $\rhobar$ is unramified above~$p$, and that certain other conditions are satisfied.
We prove that the minimal deformation ring of~$\rhobar$ coincides with the Hecke algebra acting on certain Hilbert modular forms in parallel weight one.
\end{enumerate}

We now elaborate more on these results.

\subsection{Representation theory results}\label{sec:intro-rep}

Let $H \lhd G$ be an index~$2$ subgroup such that there is a character $\chibar: H \to \FF^\times$ and $\rhobar$ is the induction of $\chibar$ from $H$ to~$G$ (these exist by the definition of dihedral representations \ref{defi:dihedral}).
Since $\rhobar$ is assumed to be absolutely irreducible, the projective image of $\rhobar$ is a dihedral group.
If the projective image of $\rhobar$ is a non-abelian dihedral group, then there is a unique index $2$ subgroup $H \lhd G$ such that $\rhobar \simeq \Ind_H^G (\chibar)$.
On the other hand, if the projective image is $\ZZ/2\ZZ \times \ZZ/2\ZZ$, then there are three such index $2$ subgroups of $G$.

As initiated by Boston \cite{boston} our analysis of deformations of~$\rhobar$ will be through actions on pro-$p$ groups, as follows.
Let $R$ be a complete Noetherian local $W(\FF)$-algebra with residue field~$\FF$.
Let $\rho: G \to \GL_2(R)$ be a lift of~$\rhobar$ and define the pro-$p$ group~$\Gamma_\rho$
by the following diagram with exact rows:
\begin{equation}\label{eq:seq}
 \xymatrix@=.5cm{
0 \ar@{->}[r]& \Gamma_\rho\ar@{->}[r] & \Image(\rho)\ar@{->}[r] & \rhobar(G) \ar@{->}[r] & 0 \\
0 \ar@{->}[r]& \Gamma_\rho\ar@{->}[r]\ar@{=}[u] & \Image(\rho|_H)\ar@{->}[r]\ar@{^(->}[u] & \rhobar(H)\ar@{->}[r] \ar@{->}@/^1pc/[l]^(.5){s} \ar@{^(->}[u] & 0 \\
}\end{equation}
Let $\Gbar^\ad$ be the image of the adjoint representation of~$\rhobar$ and $\Hbar^\ad$ the image of its restriction to~$H$.
As indicated, the lower sequence in \eqref{eq:seq} always splits, and the upper sequence splits if $p>2$.
This gives us an action on $\Gamma_\rho$ of $\Hbar^\ad$ in all cases, and of $\Gbar^\ad$ if $p>2$ or $\Gamma_\rho$ is abelian (see Lemma~\ref{lem:split-basis}).

Our first main result is the following characterisation of dihedral deformations via the $p$-Frattini quotient $\Gamma_\rho/\Phi(\Gamma_\rho)$ of~$\Gamma_\rho$,
{\it i.e.} the largest continuous quotient of $\Gamma_\rho$ that is an elementary abelian $p$-group (see \S\ref{subsec:frat}).

\begin{thm}\label{thm:inf-dih}
The following statements are equivalent:
\begin{enumerate}[(i)]
\item \label{item:id:i} There is a lift $\chi: H \to R^\times$ of $\chibar$ such that $\rho$ is equivalent to $\Ind_H^G(\chi)_R$,
the induction of $\chi$ from $H$ to~$G$.
\item \label{item:id:ii} The action of $\Hbar^\ad$ on $\Gamma_\rho$ is trivial (and hence $\rho(H) \cong \Gamma_\rho \times s(\Image(\rhobar|_H))$).
\item \label{item:id:iii} The action of $\Hbar^\ad$ on $\Gamma_\rho/\Phi(\Gamma_\rho)$ is trivial.
\end{enumerate}
\end{thm}

Concerning the final item, we provide the full list of simple $\FF_p[\Hbar^\ad]$-modules and $\FF_p[\Gbar^\ad]$-modules
that can occur in $\Gamma_\rho/\Phi(\Gamma_\rho)$ in Corollary~\ref{cor:fratt-ad}.
This theorem is applied to infinitesimal deformations that are defined as follows.

\begin{defi}\label{def:inf}
Let $\rho:G \to \GL_2(R)$ be a deformation of~$\rhobar$ as above.
Write $\rho_\inft := \pi \circ \rho$ for $\pi: \GL_2(R) \to \GL_2(R/(\fm_R^2,p))$, where $\fm_R$ is the maximal ideal of~$R$.
We say that $\rho$ is {\em infinitesimal} if $\rho = \rho_\inft$, i.e.\ if $\fm_R^2=0$ and $pR=0$.
\end{defi}
It would also be justified to speak of {\em first-order deformations}; however, we stick to the above terminology.

For the sequel we impose that the profinite group~$G$ satisfies Mazur's finiteness condition $\Phi_p$ (see \cite[\S1.1]{M}).
In that case, there exists a universal deformation
$$\rho^\univ: G \to \GL_2(R^\univ),$$
where $R^\univ$ is the universal deformation ring of $\rhobar$.
The existence of both $R^\univ$ and $\rho^\univ$ follows from \cite[Proposition 1]{M}. Moreover, it also implies that $R^\univ$ is a complete Noetherian local $W(\FF)$-algebra with residue field~$\FF$.
Write $\rho^\univ_\inft := (\rho^\univ)_\inft$, as well as $\Gamma^\univ := \Gamma_{\rho^\univ}$ and $\Gamma_\inft^\univ := \Gamma_{\rho_\inft^\univ}$.
In this notation, we find the following description of the $p$-Frattini quotient associated with the universal deformation of~$\rhobar$.

\begin{cor}\label{cor:univ-inft}
$\Gamma^\univ_\inft \cong \Gamma^\univ/\Phi(\Gamma^\univ)$.
\end{cor}

Our second main result states that the universal deformation of~$\rhobar$ is dihedral if and only if all infinitesimal deformations are.

\begin{thm}\label{thm:inf-dih-univ}
\begin{enumerate}[(a)]
\item The following statements are equivalent:
\begin{enumerate}[(i)]
\item \label{item:idu:i} $\rho^\univ$ is dihedral.
\item \label{item:idu:ii} Any deformation $\rho: G \to \GL_2(\FF[X]/(X^2))$ of $\rhobar$ is dihedral.
\item \label{item:idu:iii} Any infinitesimal deformation of $\rhobar$ is dihedral.
\item \label{item:idu:iv} $\rho^\univ_\inft$ is dihedral.
\end{enumerate}
\item If the conditions in~(a) are satisfied, then $R^\univ$ is isomorphic to the universal deformation ring $R^\univ_\chibar$ of~$\chibar$,
as computed in Proposition~\ref{prop:univ-char}.
\end{enumerate}
\end{thm} 

The main step in the proof of the theorem is to realise any group extension of $\Image(\rhobar)$ by an $\FF_p[\Gbar^\ad]$-module that can occur
in the $p$-Frattini quotient of some $\Gamma_\rho$ as the image of an infinitesimal deformation of~$\rhobar$ (see Proposition~\ref{prop:infi-lift}).
We also include an alternative proof of Theorem~\ref{thm:inf-dih-univ} using Generalised Matrix Algebras, which also works with coefficients in
a finite extension of $\QQ_p$ (see Theorem~\ref{thm:altpf}).

\subsection{Number theory results}\label{sec:intro-nt}

Let $K$ be a number field, let $G = G_K = \Gal(\Kbar/K)$ be its absolute Galois group, and let $\rhobar = \Ind_H^G(\chibar)$ be as before.
For a finite set $S$ of places of~$K$, denote by $\rho_S^\univ$ the universal deformation of~$\rhobar$ relatively unramified outside~$S$. 
Denote the \emph{constant determinant} counterpart of~$\rho_S^\univ$ by $(\rho_S^\univ)^0$.
So $\det((\rho_S^\univ)^0)$ is the Teichm\"uller lift of $\det \circ \rhobar$.

We need to introduce some further notation.
Let $S_\infty$ be the set of all archimedean places of~$K$, let $S_p$ be the set of all places above~$p$ and let $S_0$ be the set of
finite places explicitly defined in terms of~$\rhobar$ in section~\ref{sec:nt}.
Denote by $\chibar^\sigma$ the conjugate character by any $\sigma \in G \setminus H$ and by $I(\chibar/\chibar^\sigma)$ the induction
of $\chibar/\chibar^\sigma$ from $H$ to $G$ defined over its field of definition.
Furthermore, let $M^\ad$ be the fixed field under the image of the adjoint representation of~$\rhobar$, that is, $\Gal(M^\ad/K) = \Gbar^\ad$.
Denote by $A(M^\ad)$ the class group of~$M^\ad$.
We can now state our main result in this set-up, the representation-theoretic backbone of which is Theorem~\ref{thm:inf-dih}.

\begin{thm}\label{thm:nt-dih}
Let $S$ be a finite set of places of~$K$ such that
$$ S_\infty \subseteq S, \;\;\; S \cap S_p = \emptyset, \textnormal{ and } S \cap S_0 = \emptyset.$$
Assume also that the following condition holds:
$$\Hom_{\FF_p[\Gbar^\ad]}(A(M^{\ad})/p A(M^{\ad}),I(\chibar/\chibar^{\sigma})) =0.$$
If $p=2$, assume in addition that $M^\ad$ is totally imaginary.

Then $\rho^{\univ}_S$ is a dihedral deformation of~$\rhobar$.
\end{thm}

Note that in the basic case $S=S_\infty$ the set $S_0$ does not play any role, and the theorem essentially follows from Theorem~\ref{thm:inf-dih}.
More generally, the set $S_0$ takes care of the maximal elementary abelian $p$-extension unramified outside~$S$ of~$M^\ad$, in the sense that the representation $I(\chibar/\chibar^{\sigma})$ cannot occur in the corresponding Galois group viewed as
$\FF_p[\Gbar^\ad]$-module. The definition of the set $S_0$ is thus rooted in global class field theory.

In Remark~\ref{rem:nt-dih} (see also the Remarks \ref{ramrem}, \ref{classgrprem}, \ref{prem}, \ref{rem:tot-im}), we discuss in how far the hypotheses imposed in Theorem~\ref{thm:nt-dih} are necessary for the conclusion to hold.

In Corollaries \ref{corstr} and~\ref{maincor}, the structure of the universal deformation ring and its variant with `constant determinant' are computed (the latter only for $p>2$) under the assumptions of Theorem~\ref{thm:nt-dih}. The main point is that all dihedral deformations of $\rhobar$ are inductions of $1$-dimensional deformations of the character~$\chibar$.

\subsection{Application to the Boston--Fontaine--Mazur Conjecture}

Recall that Boston's strengthening of the unramified Fontaine--Mazur conjecture (\cite[Conjecture $2$]{bos}) states the following (see \cite{AC} as well):
\begin{quote}{\em Let $N$ be a number field, $\FF$ be a finite field of characteristic $p$ and $\rho : G_N \to \GL_n(\FF)$ be a continuous absolutely irreducible Galois representation. Let $S$ be a finite set of primes of $N$ not containing any prime of $N$ lying above $p$. Then the universal deformation of $\rho$ relatively unramified outside $S$ (defined in the same way as in the remark after Lemma~\ref{deflem}) has finite image.}
\end{quote}

\begin{cor}\label{cor:FM}
Let $S$ be a finite set of primes of~$K$. If the conditions given in Theorem~\ref{thm:nt-dih} hold, then Boston's strengthening of the unramified Fontaine--Mazur conjecture is true for the tuple $(K,S,\rhobar)$.
\end{cor}

As an illustration of Corollary~\ref{cor:FM}, we specialise it to a couple of examples which we describe now. Let $\FF_{81}$ be the degree $4$ extension of~$\FF_3$.
The number field $L := \QQ(\sqrt{-3},\sqrt{-239}) = \QQ(\sqrt{-3},\sqrt{717})$ has class number $15$ (see \cite[\href{http://www.lmfdb.org/NumberField/4.0.514089.1}{Global Number Field 4.0.514089.1}]{lmfdb}).
Let $M$ be its maximal unramified abelian $5$-extension. Note that the class number of both $\QQ(\sqrt{717})$ and $\QQ(\sqrt{-3})$ is~$1$.
Therefore, $M$ is a Galois extension of both $\QQ(\sqrt{717})$ and $\QQ(\sqrt{-3})$ with $\Gal(M/\QQ(\sqrt{-3})) \simeq \Gal(M/\QQ(\sqrt{717})) \simeq D_5$.
In these cases, Corollary~\ref{cor:FM} gives us the following results:

\begin{cor}\label{cor:exm}
Let $\chi : \Gal(M/\QQ(\sqrt{-3},\sqrt{-239})) \to \FF_{81}^*$ be a non-trivial continuous character.
\begin{enumerate}[(a)]
\item\label{cor:exm:a} Let $\rhobar_1 : G_{\QQ(\sqrt{717})} \to \GL_2(\FF_{81})$ be the representation $\Ind_{\Gal(M/L)}^{\Gal(M/\QQ(\sqrt{717}))}(\chi)$.
Let $S$ be a finite set of primes of $\QQ(\sqrt{717})$ such that $S_{\infty} \subseteq S$, $S$ does not contain any prime above~$3$,
and all the finite primes contained in $S$ are split in $L$ but not completely split in $M$.
Then the universal deformation of $\rhobar_1$ relatively unramified outside $S$ is dihedral and has finite image.
\item\label{cor:exm:b} Let $\rhobar_2 : G_{\QQ(\sqrt{-3})} \to \GL_2(\FF_{81})$ be the representation $\Ind_{\Gal(M/L)}^{\Gal(M/\QQ(\sqrt{-3}))}(\chi)$.
Let $S$ be a finite set of primes of $\QQ(\sqrt{-3})$ such that $S_{\infty} \subseteq S$, $S$ does not contain any prime above~$3$,
and all the finite primes contained in $S$ are split in~$L$ but not completely split in $M$.
Then the universal deformation of $\rhobar_2$ relatively unramified outside $S$ is dihedral and has finite image.
\end{enumerate}
\end{cor}

Note that Boston's conjecture has been proved by Allen and Calegari for a certain class of representations of the absolute Galois groups of totally real fields (see \cite[Corollary 3]{AC}).
However, the two cases considered above do not satisfy the hypotheses of \cite[Corollary 3]{AC} and, hence, they give us new evidence towards Boston's strengthening of the unramified Fontaine--Mazur conjecture.

In section~\ref{sec:examples}, we report on some computer calculations that we carried out to obtain examples when the universal relatively
unramified deformation of $\rhobar$ is dihedral, and others when this is not the case.
Those examples for which the universal relatively unramified deformation is dihedral also provide explicit examples
in favour of Boston's strengthening of the unramified Fontaine--Mazur conjecture.

\subsection{Application to a question of Greenberg and Coleman}

A famous question due to R.~Greenberg asks when the $p$-adic Galois representation attached to a $p$-ordinary cuspidal eigenform $f$ of weight $k \geq 2$ is a sum of two characters when restricted to a decomposition group at $p$ (see \cite[Question 1]{GV04}).
An equivalent form of this question due to Coleman can also be found in \cite{Col} (see \cite[Section 1]{CWE} for more details).
There has been a lot of work on this question in the past; see \cite[Section 1]{CWE} for a brief summary of the work centred around this question.
An answer to this question was found under certain hypotheses in \cite{CWE} (\cite[Theorem 1.3.1, Corollary 1.3.2]{CWE}). 
We prove a similar theorem:

\begin{thm}\label{thm:CM2}
Assume Set-up~\ref{setup:app} and suppose $\Hom_{\FF_p[\Gbar^\ad]}(A(M^{\ad})/pA(M^{\ad}) ,  I(\chibar/\chibar^{\sigma})) = 0$ (cf.\ Theorem~\ref{thm:nt-dih}).
Let $f$ be a classical newform of tame level $\Gamma_1(M)$ such that
\begin{enumerate}
\item $\rhobar_f \simeq \rhobar$,
\item $\rho_f|_{G_{\QQ_p}}$ is a sum of two characters,
\item if $\ell | M$ and $\rhobar|_{G_{\QQ_{\ell}}} = \eta \oplus \psi$, then $\eta\psi^{-1} \neq \omega^{(\ell)}_p, (\omega^{(\ell)}_p)^{-1}$,
where $\omega^{(\ell)}_p$ is the mod $p$ cyclotomic character of $G_{\QQ_{\ell}}$,
\item if $\ell | M$, $p | \ell-1$ and $\rhobar|_{G_{\QQ_{\ell}}}$ is reducible, then one of the following holds:
\begin{enumerate}
\item $\ell$ is split in $L$,
\item $\ell$ is not split in $L$, $\det(\rho_f|_{I_{\ell}}) = \widehat{\det(\rhobar|_{I_{\ell}})}$ and $\ell \nmid M/N$.
\end{enumerate}
\end{enumerate}
Then $f$ has CM by $L$.
\end{thm}

Note that Theorem~\ref{thm:CM2} implies \cite[Corollary $1.3.2$]{CWE} (see Remark~\ref{rem:CM} for more details).
To prove Theorem~\ref{thm:CM2}, we identify a quotient $G^{p-\ab}_S$ of $G_{\QQ}$ such that if $f$ is an eigenform satisfying the conditions of Theorem~\ref{thm:CM2}, then $\rho_f$ factors through $G^{p-\ab}_S$. 
Under the hypotheses of Theorem~\ref{thm:CM2}, we prove that the universal deformation of $\rhobar$ for the group $G^{p-\ab}_S$ is dihedral by verifying the criteria proved in Theorem~\ref{thm:inf-dih-univ}.
The theorem follows directly from this.

\subsection{Modularity results}\label{sec:intro-modularity}

We apply our number theoretic results towards a comparison between a minimal universal deformation ring and a Hecke algebra in parallel weight one.
The main point is that, when $K$ is totally real, irreducible totally odd induced representations of finite order complex-valued characters are afforded by cuspidal Hilbert modular forms of parallel weight one that are induced from the corresponding Hecke characters.

We keep the objects from the previous subsection and impose several additional hypotheses that are natural in view of the application to Hilbert modular forms and the previous results:
\begin{enumerate}
\item $p>2$.
\item $K$ is totally real.
\item The character $\chibar$ is such that $\rhobar$ is totally odd.
\item $\rhobar$ is unramified at all places above~$p$.
\item If $\rhobar$ is ramified at a prime $\ell$ of $K$ and $\rhobar|_{G_{K_{\ell}}}$ is not absolutely irreducible, then $\dim((\rhobar)^{I_{\ell}}) =1$ where $(\rhobar)^{I_{\ell}}$ denotes the subspace of $\rhobar$ fixed by the inertia group $I_{\ell}$ at $\ell$.
\item $\Hom_{\FF_p[\Gbar^\ad]}(A(M^{\ad})/p A(M^{\ad}),I(\chibar/\chibar^{\sigma})) =0$ (cf.\ Theorem~\ref{thm:nt-dih}).
\end{enumerate}

We shall restrict to {\em minimal deformations} (see Definition~\ref{mindef}), defined in the same way as in \cite[Definition $3.1$]{CG}.
We find that there is an explicit choice for the set of places~$S$ such that $(\rho_S^\univ)^0$ is dihedral and $R^\minim = (R_S^\univ)^0$ (see Proposition~\ref{prop:min-univ}),
where $R^\minim$ is the universal minimal deformation ring of~$\rhobar$ and $(R_S^\univ)^0$ is the ring underlying $(\rho_S^\univ)^0$.

Using that $(\rho_S^\univ)^0$ is dihedral, from Proposition~\ref{prop:Rmodular} we obtain a quotient $\TT_\dih$ of the anemic $W(\FF)$-Hecke algebra acting on parallel weight one Hilbert modular forms over~$K$,
which comes equipped with a Galois representation $\rho_\dih: G_K \to \GL_2(\TT_\dih)$ and a ring homomorphism $\psi : \TT_\dih \to (R_S^\univ)^0$ such that $\psi \circ \rho_\dih = (\rho_S^\univ)^0$. The determinant of $\rho_\dih$ is the Teichm\"uller lift of $\det \circ \rhobar$.
It turns out that $\rho_\dih$ is a minimal deformation of~$\rhobar$ and this leads to our main modularity result.

\begin{thm}\label{minthm}
The map $R^\minim \to \TT_\dih$ induced from the minimal deformation property of $\rho_\dih$ is an isomorphism.
\end{thm}

\subsection{Notation and conventions}

We summarise some notation and conventions to be used throughout the paper. More notation is introduced during the text.

For a finite field~$\FF$, denote by $W(\FF)$ the ring of Witt vectors of~$\FF$.
Let $\calC$ be the category of local complete Noetherian $W(\FF)$-algebras $R$ with residue field $\FF$.
The Teichm\"uller lift of an element $x \in \FF$ to $W(\FF)$ (and to any $W(\FF)$-algebra) will be denoted by a hat:~$\hat{x}$.
All representations are assumed to be continuous without explicit mention of this.
For a local ring~$R$, denote by $\fm_R$ its maximal ideal.

Specific objects that are used without explicit mention in the statements of propositions and theorems are collected in `set-ups'.
\begin{setup}
In the entire article, $p$ will denote a fixed prime number and $\FF$ a finite field of characteristic~$p$.
\end{setup}

\subsection{Acknowledgements}
The authors would like to thank Carl Wang-Erickson for useful discussions and encouraging them to include the equi-characteristic case of Theorem~\ref{thm:inf-dih-univ} as well.
They also thank Tarun Dalal for pointing out an error in a previous version of this article.
They also thank the referee for insightful and helpful comments and suggestions.

The majority of this work was done when S.~D. was a postdoc at the University of Luxembourg.
G.~W.\ acknowledges support by the IRP AMFOR at the University of Luxembourg.
This research is also supported by the Luxembourg National Research Fund INTER/ANR/18/12589973 GALF.

\section{Representation theory}\label{sec:rep}

In this section, we develop and prove the representation theory results outlined in section~\ref{sec:intro-rep}.

\subsection{Explicit universal deformations of characters}

Since dihedral representations are induced from characters, we first include a treatment of the universal
deformation of a character.
It can be derived from Mazur's fundamental paper~\cite[\S1.4]{M}, but for the sake of completeness, we include a proof, which is quite simple and explicit.

For $r \in \NN$ and $n$-tuples of positive integers $(e_1,e_2,\dots,e_n)$ we introduce the piece of notation
$$ \calU_{W(\FF),r,(e_1,e_2,\dots,e_n)} := W(\FF)[[X_1,\dots,X_{n+r}]]/((1+X_1)^{p^{e_1}} -1, \dots, (1+X_{n})^{p^{e_n}}-1).$$

\begin{prop}\label{prop:univ-char}
Let $H$ be a profinite group.
We assume that the pro-$p$ group $P = \prod_{i=1}^n \ZZ/p^{e_i}\ZZ \times \ZZ_p^r$ with $e_i \ge 1$ for all $1 \le i \le n$
is the maximal continuous abelian pro-$p$ quotient of~$H$. Let $g_1,\dots,g_{n+r}$ be generators of~$P$ such that
$g_i$ topologically generates $\ZZ/p^{e_i}\ZZ$ for $1 \le i \le n$ and $\ZZ_p$ for $n+1 \le i \le n+r$.

Let $\chibar: H \to \FF^\times$ be a character and denote by $\chibarhat: H \to W(\FF)^\times$ its Teichm\"uller lift.
Define the character
$\psi^\univ: H \to \calU_{W(\FF),r,(e_1,e_2,\dots,e_n)}$
as the composition of the projection $H \twoheadrightarrow P$ and the group monomorphism 
$P \to \big(\calU_{W(\FF),r,(e_1,e_2,\dots,e_n)} \big)^\times$ sending $g_i$ to $1+X_i$ for $i=1,\dots,n+r$.
Also define the {\em universal character}
$$ \chi^\univ := \psi^\univ \cdot \chibarhat.$$

Then $ \calU_{W(\FF),r,(e_1,e_2,\dots,e_n)}$ is the universal deformation ring of~$\chibar$
in the category~$\calC$ and $\chi^\univ$ is the universal deformation.
\end{prop}

\begin{proof}
It is a simple check that $\chi^\univ$ is well-defined and indeed a deformation of~$\chibar$.
Let now $R$ be in~$\calC$ and $\chi: H \to R^\times$ a deformation of~$\chibar$.
We set $\psi := \chi \cdot \chibarhat^{-1}$. As the reduction of~$\psi$ is trivial, its image is a pro-$p$ group and thus a quotient of~$P$.
We write this as $\psi: H \twoheadrightarrow P \overset{\pi}{\twoheadrightarrow} \Image(\psi) \subseteq R^\times$.

We define the $W(\FF)$-algebra homomorphism
$$ W(\FF)[[X_1,\dots,X_{n+r}]] \to R, \;\;\; X_1 \mapsto \pi(g_1) - 1, \dots, X_{n+r} \mapsto \pi(g_{n+r}) - 1.$$
The elements $(1+X_1)^{p^{e_1}} -1, \dots, (1+X_n)^{p^{e_n}}-1$ are clearly in its kernel so that we obtain
a $W(\FF)$-algebra homomorphism 
$$ \phi: \calU_{W(\FF),r,(e_1,e_2,\dots,e_n)} \to R.$$

The commutativity of the diagram
$$ \xymatrix@=.5cm{
&&&  \big(\calU_{W(\FF),(e_1,e_2,\dots,e_n)}\big)^\times \ar@{->}[dd]^\phi\\
H \ar@{->>}[rr] \ar@{->}@/^1pc/[urrr]^(.35){\psi^\univ} \ar@{->}@/_1pc/[drrr]_\psi && P  \ar@{->}[ur] \ar@{->}[dr]\\
&&& R^\times \\
}$$
is clear. It implies $\chi = \phi \circ \chi^\univ$. The uniqueness of $\phi$ with this property is clear.
All this together shows the universality of $(\calU_{W(\FF),r,(e_1,e_2,\dots,e_n)},\chi^\univ)$. 
\end{proof}

\subsection{Dihedral representations}

We start by clarifying what we mean by induced and dihedral representations in the special cases we need.

\begin{defi}\label{defi:dihedral}
Let $G$ be a profinite group and $R$ a topological ring. A representation $\rho: G \to  \GL_2(R)$ is called {\em dihedral}
if there is an open index-$2$ subgroup $H \lhd G$ and a character $\chi: H  \to R^\times$ such that $\rho$ is equivalent to
$\Ind_H^G(\chi)_R$, where the free $R$-module of rank~$2$
$$ \Ind_H^G(\chi)_R = \{f: G \to R \textnormal{ map }|\; \forall\,g \in G,\forall\,h \in H: f(hg) = \chi(h)f(g)\}$$
is the {\em induced representation of~$\chi$ from $H$ to~$G$} equipped with the left $G$-action via
$(\tilde{g}.f)(g) = f(g\tilde{g})$ for $g,\tilde{g} \in G$.
\end{defi}

We stress that in the definition we ask the character $\chi$ to be defined over~$R$. This choice may not be standard, but can always
be achieved by extending~$R$. It simplifies working matricially with induced representations.

For the sake of being explicit and making certain proofs more transparent, we quickly describe a matrix representation of $\rho=\Ind_H^G(\chi)_R$.
Let us write $G = H \sqcup \sigma H$ and put $\chi^\sigma(h) = \chi(\sigma h\sigma^{-1})$ for $h \in H$.
Then with respect to a natural choice of basis, for $h \in H$, we have
\begin{equation}\label{eq:dih-matrix}
   \rho(h)        = \mat {\chi(h)}00{\chi^\sigma(h)} \textnormal{ and }
   \rho(\sigma h) = \mat 0{\chi^\sigma(h)}{\chi(h)\chi(\sigma^2)}0.
\end{equation}

The name {\em dihedral representation} is justified because an irreducible representation $\rhobar: G \to \GL_2(\FF)$ with $\FF$ a finite field
is dihedral if and only if its projective image is a dihedral group (after possibly replacing $\FF$ by a finite extension).

For a representation $\rho:G \to \GL_2(R)$ we define the adjoint representations
$\ad(\rho)_R$ and $\ad^0(\rho)_R$ as the representations given by the conjugacy of $\rho$ on~$M_2(R)$ and $M_2^0(R)$, respectively,
where $M_2(R)$ are the $2\times 2$-matrices with coefficients in~$R$ and $M_2^0(R)$ is its $R$-submodule consisting of the matrices
having trace~$0$.

From now on, we assume the following set-up.
\begin{setup}\label{setup:one}
Let $G$ be a profinite group, $H \lhd G$ an open subgroup of index~$2$ and $\sigma \in G \setminus H$.
\end{setup}

\begin{defi}\label{defi:modules}
\begin{enumerate}[(a)]
\item For an extension of topological rings $R \subseteq R'$,
a character $\epsilon: G \to R^\times$ and a character $\chi: H \to R^\times$, we make the following definitions:
\begin{itemize}
\item $C(\epsilon)_{R'}$ is $R'$ with $G$-action through~$\epsilon$; in particular, $C(1)_{R'}$ is the trivial module;
\item $C(\chi)_{R'}$ is $R'$ with $H$-action through~$\chi$;
\item $N_{R'}$ is ${R'}^2$ with trivial~$H$-action and $\sigma$ acting by swapping the two standard basis vectors;
\item $I(\chi)_{R'} = \Ind_H^G(\chi)_{R'}$, as described in Definition~\ref{defi:dihedral}.
\end{itemize}
\item In the case of finite fields of characteristic~$p>0$, we sometimes drop minimal fields of definition from the notation.
In particular, we write $C(1) := C(1)_{\FF_p}$, $C(\epsilon) := C(\epsilon)_{\FF_p}$ if $\epsilon$ is at most quadratic,
$N := N_{\FF_p}$ and $I(\chi) := I(\chi)_{F_0}$ if $F_0$ is the extension of~$\FF_p$ generated by the coefficients of all occurring
characteristic polynomials.
\end{enumerate}
\end{defi}

\begin{lem}\label{lem:dihedral}
Let $R$ be a topological ring and let $\rho\cong \Ind_H^G(\chi)_R$ for some character $\chi: H \to R^\times$.
Choose a basis of $R^2$ such that as in \eqref{eq:dih-matrix} under this basis, $\rho(h)$ is diagonal for all $h \in H$.
Then, under the choice of this basis, the map
$$\ad(\rho)_R \to N_R \oplus I(\chi/\chi^\sigma)_R, \;\;\; \mat abcd \mapsto \vect ad \oplus \vect b{c/\chi(\sigma^2)}$$
is an isomorphism of $R[G]$-modules.
Moreover, one has an isomorphism of $R[G]$-modules
$$\ad^0(\rho)_R \cong C(\epsilon)_R \oplus I(\chi/\chi^\sigma)_R,$$
where $\epsilon: G \twoheadrightarrow G/H \twoheadrightarrow \{\pm 1\} \subseteq R^\times$.
Furthermore, if $2$ is invertible in~$R$, then $N_R$ is isomorphic to $C(1)_R \oplus C(\epsilon)_R$ as $R[G]$-modules.
Finally, the exact sequence of $R[G]$-modules
$$ 0 \to \ad^0(\rho)_R \to \ad(\rho)_R \xrightarrow{\tr} C(1)_R \to 0$$
is split if $2$ is invertible in~$R$ with split $r \mapsto \mat{r/2}00{r/2}$.
\end{lem}

\begin{proof}
These are elementary calculations. 
\end{proof}

\begin{setup}\label{setup:two}
In addition to Set-up~\ref{setup:one}, let $\chibar: H \to \FF^\times$ be a character such that $\chibar \neq \chibar^\sigma$, where $\chibar^\sigma(h) = \chibar(\sigma h\sigma^{-1})$ for $h \in H$ is the conjugate character.
Let $\rhobar: G \to \GL_2(\FF)$ be $\Ind_H^G(\chibar)_\FF$. By the assumption $\chibar \neq \chibar^\sigma$, the representation $\rhobar$
is absolutely irreducible.
We also use the following pieces of notation:
$$ \Gbar = \rhobar(G), \;\;\Hbar = \rhobar(H),\;\;
\Gbar^\ad = \ad(\rhobar)(G), \;\;\Hbar^\ad = \ad(\rhobar)(H), \;\; C := \ker(\Gbar \twoheadrightarrow \Gbar^\ad).$$
\end{setup}

Note that $\Gbar^\ad$ is a quotient of~$\Gbar$ and $\Hbar^\ad$ is a quotient of~$\Hbar$. Furthermore,
$$ C=\ker(\Gbar \twoheadrightarrow \Gbar^\ad) =  \ker(\Hbar \twoheadrightarrow \Hbar^\ad)
\subseteq \{ \mat a00a \;|\; a \in \FF^\times\}$$
and
$\Gbar^\ad$ is isomorphic to the image of $\Ind_H^G(\chibar/\chibar^\sigma)_\FF$. If $\chibar/\chibar^{\sigma} = \chibar^{\sigma}/\chibar$, then $I(\chibar/\chibar^{\sigma})_{\FF}= C(\chibar_1)_{\FF} \oplus C(\chibar_2)_{\FF}$ for some characters $\chibar_1$, $\chibar_2 : G \to \FF^{\times}$. We will keep this notation for the rest of the section.

In the sequel, we will make frequent use of the Krull--Schmidt--Azumaya theorem (\cite[(6.12)]{CR}) allowing us to express modules
over group rings of finite groups uniquely as direct sums of indecomposables.

\begin{lem}\label{lem:indec}
\begin{enumerate}[(a)]
\item\label{lem:indec:a} The simple $\FF_p[\Hbar^\ad]$-modules occurring in $\ad(\rhobar)_\FF$ are $C(1)$,
$C(\chibar/\chibar^\sigma)$, $C(\chibar^\sigma/\chibar)$. Every indecomposable $\FF_p[\Hbar^\ad]$-module is simple.
\item\label{lem:indec:b} If $p>2$, then the simple $\FF_p[\Gbar^\ad]$-modules occurring in $\ad(\rhobar)_\FF$ are $C(1)$, $C(\epsilon)$
and $I(\chibar/\chibar^\sigma)$ (or $C(\chibar_1)$ and $C(\chibar_2)$ for some characters $\chibar_1,\chibar_2:\Gbar^\ad \to \FF^\times$  if $\chibar/\chibar^\sigma = \chibar^\sigma/\chibar$).
Every indecomposable $\FF_p[\Gbar^\ad]$-module is simple.
\item\label{lem:indec:c} For $p=2$, the simple $\FF_2[\Gbar^\ad]$-modules occurring as Jordan--Hölder factors of $\ad(\rhobar)_\FF$ are $C(1)$ and $I(\chibar/\chibar^\sigma)$.
The only indecomposable non-simple $\FF_2[\Gbar^\ad]$-module whose composition factors are in $\{C(1),I(\chibar/\chibar^\sigma)\}$ is~$N$.
\end{enumerate}
\end{lem}

\begin{proof}
(a,b) Since $p \nmid \#\Hbar^\ad$ and $p \nmid \#\Gbar^\ad$ if $p>2$, by Maschke's theorem \cite[Theorem~3.14]{CR} every indecomposable module is simple.
Lemma~\ref{lem:dihedral} gives the list of occurring simple modules. Note that 
$I(\chibar/\chibar^\sigma)$ is simple if and only if $(\chibar/\chibar^\sigma)^2 \neq 1$.
Note also that we use that $C(\epsilon)_\FF \cong C(\epsilon) \otimes_{\FF_p} \FF \cong C(\epsilon)^{[\FF:\FF_p]}$
and $I(\chibar/\chibar^\sigma)_\FF \cong I(\chibar/\chibar^\sigma)^{[\FF:F_0]}$ (in the notation of Definition~\ref{defi:modules}) as $\FF_p[\Gbar^\ad]$-modules,
and similarly for the other modules and over $\FF_p[\Hbar^\ad]$.

(c) The list of simple modules from~(b) is also valid for~$p=2$.
Note that the Jordan--Hölder factors of~$N$ are all~$C(1)$.
By assumption we have $\chibar/\chibar^\sigma \neq \chibar^\sigma/\chibar$ (since $p=2$).
Let $V$ be an indecomposable non-simple $\FF_2[\Gbar^\ad]$-module the composition factors of which occur as Jordan--Hölder factors of $\ad(\rhobar)_\FF$.
We first decompose~$V$ as $\FF_2[\Hbar^\ad]$-module into
$$ V \cong C(1)^{r_1} \oplus C(\chibar/\chibar^\sigma)^{r_2} \oplus C(\chibar^\sigma/\chibar)^{r_3}.$$
This decomposition can be considered as a decomposition into simultaneous eigenspaces for the $H$-action. Note that $G$ permutes the occurring simultaneous
eigenspaces. More precisely, it stabilises $C(1)^{r_1}$ and $\sigma\big( C(\chibar/\chibar^\sigma) \big) = C(\chibar^\sigma/\chibar)$.
So $r_2=r_3$ follows and thus $V \cong C(1)^{r_1} \oplus I(\chibar/\chibar^\sigma)^{r}$ as $\FF_p[\Gbar^\ad]$-modules.
By the indecomposable non-simple assumption
$V \cong C(1)^{r_1}$, i.e.\ $V$ is $\FF_2^{r_1}$ with trivial~$H$-action and an involutive action by~$\sigma$.
Due to the indecomposability, in the Jordan normal form of~$\sigma$ on~$V$ there can only be a single Jordan block.
This block has to have size~$1$ or~$2$ as otherwise the order of $\sigma$ would be larger than~$2$.
As $V$ is non-simple, the block size has to be~$2$ and $V$ is thus isomorphic to~$N$. 
\end{proof}

Although we formulate the following corollary for all primes~$p$, it is only non-trivial for~$p=2$.

\begin{cor}\label{cor:indec}
Any indecomposable $\FF_p[\Gbar^\ad]$-module the composition factors of which are among those of $\ad(\rhobar)_\FF$
is a submodule of $\ad(\rhobar)_\FF$.
\end{cor}

\begin{proof}
This follows immediately from Lemma~\ref{lem:indec}. 
\end{proof}

Let $R \in \calC$, $\rho: G \to \GL_2(R)$ be a lift of~$\rhobar$ and $\Gamma_\rho$ be defined by the diagram~\eqref{eq:seq}.

\begin{lem}\label{lem:split-basis}
\begin{enumerate}[(a)]
\item The lower exact sequence in~\eqref{eq:seq} splits, as indicated in the diagram.

\item There is an $R$-basis of $\rho$ such that for all $h \in H$ one has
$$s\circ\rhobar(h) = \mat{\chibarhat(h)}00{\chibarhat^\sigma(h)},$$
where the hat indicates the Teichm\"uller lift.
In particular, $s \circ \rhobar(h)$ is scalar if $\rhobar(h)$ is scalar.

Thus the conjugation action of $\Hbar$ on $\Gamma_\rho$ via~$s$ descends to~$\Hbar^\ad$.

\item If $p>2$, the upper sequence in~\eqref{eq:seq} splits, leading to a conjugation action of $\Gbar$ on~$\Gamma_\rho$,
which descends to $\Gbar^\ad$.

\item If $\Gamma_\rho$ is abelian (for instance, if $\fm_R^2 = 0$), then via choices of preimages
the group $\Gbar$ acts on~$\Gamma_\rho$ via conjugation, and this action descends to~$\Gbar^\ad$.

\end{enumerate}
\end{lem}

\begin{proof}
The splitting of the exact sequences in (a) and (c) follows from the theorem of Schur--Zassen\-haus \cite[(18.1)]{Aschbacher}
since the group orders of $\Hbar$ (resp.\ $\Gbar$) are coprime to the order of~$\Gamma_\rho$.

(b) By its explicit description, $\FF^2$ has a basis consisting of simultaneous eigenvectors for~$\Hbar$;
the eigenvalues are distinct for some matrices due to the assumption that $\rhobar$ is irreducible. Let $a,b$ be two
such distinct eigenvalues, occurring for some $\rhobar(h)$. The order $n$ of $\rhobar(h)$ is not divisible by~$p$.
Hence the polynomial $X^n-1$ annihilates $\rhobar(h)$ (and also $s \circ \rhobar(h)$) and factors into $n$ distinct coprime
factors over~$\FF$. Then so it does over~$R$ i.e.\
$$X^n-1= (X-\hat{a})(X-\hat{b})f(X)$$
for some $f \in R[X]$. Denote by $\overline{f} \in \FF[X]$ the reduction of~$f$ modulo~$\fm_R$.
As $\det(\overline{f}(\rhobar(h)))$ is invertible in~$\FF$, also $f(s \circ \rhobar(h))$ is invertible.
Consequently $(X-\hat{a})(X-\hat{b})$ annihilates $s\circ \rhobar(h)$. As the two polynomials $(X-\hat{a})$ and $(X-\hat{b})$
are coprime, the representation space $R^2$ of~$\rho$ is the direct sum of the eigenspaces of $s \circ \rhobar(h)$
for the eigenvalues~$\hat{a}$ and~$\hat{b}$.
By Nakayama's lemma, each of these eigenspaces is a non-trivial quotient of~$R$ and each eigenspace is generated by one element over $R$ as this is the case over~$\FF$.
This leads to a surjection $R^2 \to R^2$ of Noetherian modules, which is hence an isomorphism, showing that
the each eigenspace is free of rank~$1$ as $R$-module.

(c) The action descends to $\Gbar^\ad$ because the kernel $C$ of $\Gbar \to \Gbar^\ad$ acts through scalar matrices due to~(b).

(d) is clear as by part~(b), the action descends to $\Gbar^\ad$. 
\end{proof}

\subsection{Characterisation of dihedral representations by Frattini quotients}\label{subsec:frat}
For a pro-$p$ group~$\Gamma$ we denote by $\Phi(\Gamma)$ its $p$-Frattini subgroup, that is, the closure of $\Gamma^p [\Gamma,\Gamma]$ in~$\Gamma$.
The quotient $\Gamma/\Phi(\Gamma)$ will be called the $p$-Frattini quotient. It can be characterised as the
largest continuous quotient of~$\Gamma$ that is an elementary abelian $p$-group.
Note that the $p$-Frattini subgroup is a characteristic subgroup and the actions on $\Gamma_\rho$ from Lemma~\ref{lem:split-basis} induce actions
on the $p$-Frattini quotient.

The key input for characterising dihedral deformations is the following fact from group theory.

\begin{prop}\label{prop:coprime-aut}
Let $\Gamma$ be a pro-$p$ group and let $A \subseteq \Aut(\Gamma)$ be a finite subgroup of order coprime to~$p$.
Then the natural map $A \to \Aut(\Gamma/\Phi(\Gamma))$ is injective.
\end{prop}

\begin{proof}
The version for a finite $p$-group~$\Gamma$ is proved in \cite[(24.1)]{Aschbacher}. To see the statement for pro-$p$ groups,
consider an automorphism $\alpha$ of~$\Gamma$ that is trivial on $\Gamma/\Phi(\Gamma)$. By the result for finite $p$-groups, $\alpha$
is then also trivial on any finite quotient $\Gamma'$ of~$\Gamma$ because $\Gamma'/\Phi(\Gamma')$ is a quotient of $\Gamma/\Phi(\Gamma)$.
Consequently, $\alpha$ is trivial on~$\Gamma$. 
\end{proof}

We can now prove the characterisation of dihedral representations via Frattini quotients. We continue to use the notation introduced above.

\begin{proof}[Proof of Theorem~\ref{thm:inf-dih}.]
Before starting the proof of the equivalences, let us prove the implication mentioned in item \eqref{item:id:ii}:
if the action of $\Hbar^\ad$ on $\Gamma_\rho$ is trivial, then so is the action of~$\Hbar$; as this action is by conjugation via the splitting,
$s(\Hbar)$ and $\Gamma_\rho$ commute, leading to $\rho(H) = \Gamma_\rho \times s(\Hbar)$.

Next, we apply Proposition~\ref{prop:coprime-aut}, yielding the equivalence of \eqref{item:id:ii} and \eqref{item:id:iii}.

Let us assume~~\eqref{item:id:i}. From the matricial description of~$\rho$ in \eqref{eq:dih-matrix} we see that 
$\rho(H)$ sits in the diagonal matrices and is hence abelian.
Thus the conjugation action by $\Hbar$ on $\Gamma_\rho$ is trivial, showing~\eqref{item:id:ii}.

Let us now assume~\eqref{item:id:ii}. As already seen, we then have $\rho(H) = \Gamma_\rho \times s(\Hbar)$.
We choose an $R$-basis $v_1,v_2$ as in Lemma~\ref{lem:split-basis}~(b).
For this basis of $R^2$, the matrices representing elements in $\Gamma_\rho$ have to be diagonal as well,
as any matrix commuting with a non-scalar diagonal matrix with unit entries is diagonal itself.
This implies that $\Gamma_\rho$ is an abelian pro-$p$ group.
We can thus see $\Gamma_\rho$ as being given by two characters $\psi_1,\psi_2:H \to R^\times$, i.e.\
$\rho(h) = \mat{\psi_1(h) \chibarhat(h)}00{\psi_2(h) \chibarhat^\sigma(h)}$ for $h \in H$.
Moreover, conjugation by~$\rho(\sigma)$ swaps the two simultaneous eigenvectors, proving $\psi_2 =  \psi_1^\sigma$.
The matricial description of induced representations in \eqref{eq:dih-matrix} immediately implies~\eqref{item:id:i}. 
\end{proof}

\begin{lem}\label{lem:gamma-pro-p}
Let $R \in \calC$ and $\rho: G \to \GL_2(R)$ a lift of~$\rhobar$.
As before, define $\Gamma = \Gamma_\rho = \ker\big(\Image(\rho) \twoheadrightarrow \Image(\rhobar)\big)$.
For $k \in \ZZ_{\ge 1}$, also define $\Gamma_k = \Image\big(\Gamma \hookrightarrow \Image(\rho) \twoheadrightarrow \Image(\rho \mod \fm_R^k)\big)$.

Then $\Gamma = \plim{k} \Gamma_k$ and we have $\Gbar^\ad$-equivariantly:
$$\ker\big(\Gamma_k \twoheadrightarrow \Gamma_{k-1}\big) \subseteq 1 + M_2(\fm_R^{k-1}/\fm_R^k) \underset{\sim}{\xrightarrow{1+A \mapsto A}} M_2(\FF)^{r_k},$$
where $r_k = \dim_{\FF} \fm_R^{k-1}/\fm_R^k$.
Moreover, if $\det(\rho)$ is the Teichm\"uller lift of $\det(\rhobar)$, then $\ker\big(\Gamma_k \twoheadrightarrow \Gamma_{k-1}\big)$
is contained in $M_2^0(\FF)^{r_k}$.
\end{lem}

\begin{proof}
The first statement is clear.
The inclusion in the second statement is a consequence of the fact that the kernel of the projection
$\pi_k: \GL_2(R/\fm_R^k) \twoheadrightarrow \GL_2(R/\fm_R^{k-1})$ is given by~$1 + M_2(\fm_R^{k-1}/\fm_R^k)$.
Furthermore, note that $\Gbar$ acts on $\ker\big(\Gamma_k \twoheadrightarrow \Gamma_{k-1}\big)$ for any~$k$ by conjugation with a preimage in $\Image(\rho \mod \fm_R^k)$
and that this action is independent of the choice of preimage because conjugation by $1+M_2(\fm_R/\fm_R^k)$ on $\ker\big(\Gamma_k \twoheadrightarrow \Gamma_{k-1}\big)$ is trivial.
Thus, by Lemma~\ref{lem:split-basis}, the action descends indeed to an action of~$\Gbar^\ad$.
The $\Gbar^\ad$-equivariance and the final assertion follow from simple calculations. 
\end{proof}

\begin{cor}\label{cor:fratt-ad}
The indecomposable $\FF_p[\Gbar^\ad]$-modules occurring in $\Gamma_\rho/\Phi(\Gamma_\rho)$ are submodules of
the adjoint representation $\ad(\rhobar)_\FF$, that is, they are isomorphic to $C(1)$, $C(\epsilon)$, $I(\chibar/\chibar^\sigma)$ (or $C(\chibar_1)$, $C(\chibar_2)$ if $\chibar/\chibar^{\sigma}=\chibar^{\sigma}/\chibar$)
or, if $p=2$, the unique non-trivial extension $N$ of $C(1)$ by itself.

In particular, as a $\FF_p[\Hbar^\ad]$-module, $\Gamma_\rho/\Phi(\Gamma_\rho)$ is isomorphic to $C(1)^r \oplus I(\chibar/\chibar^\sigma)^s$ (or to $C(1)^r \oplus C(\chibar/\chibar^{\sigma})^s$ if $\chibar/\chibar^{\sigma}=\chibar^{\sigma}/\chibar$)
for some $r,s \in \NN$, and, thus, the $\FF_p[\Hbar^\ad]$-action on $\Gamma_\rho/\Phi(\Gamma_\rho)$ is trivial if and only if $s=0$.
\end{cor}

\begin{proof}
This follows from Lemmata \ref{lem:gamma-pro-p}, \ref{lem:indec} and Corollary~\ref{cor:indec}. 
\end{proof}

Note that the conclusion is in terms of $\FF_p[\Gbar^\ad]$-representations, not $\FF[\Gbar^\ad]$-repre\-sen\-tations,
because it is not clear (and usually wrong) that $\Gamma_\rho/\Phi(\Gamma_\rho)$ has the structure of $\FF$-vector space.

\subsection{The infinitesimal quotient of the universal representation}

The previous computations are valid for all representations. In this subsection we specialise to the universal representation
because for it we can replace the Frattini quotient by an infinitesimal deformation.

\begin{lem}
Let $\rho:G \to \GL_2(R)$ be an infinitesimal deformation of~$\rhobar$.
Then $\fm_R$ is an $\FF$-vector space of some finite dimension $r$ and
we have the inclusion of $\FF_p[\Gbar^\ad]$-modules $\Gamma_\rho \subseteq \ad(\rhobar)_\FF^r$.
\end{lem}

\begin{proof}
The kernel $\Gamma_\rho$ of reduction modulo~$\fm_R$ clearly sits in $1+M_2(\fm_R)$, proving the result. 
\end{proof}

We now prove a converse of Corollary~\ref{cor:fratt-ad}.
In the case $p=2$, the upper exact sequence in~\eqref{eq:seq} need not split. This is taken
into account in the following proposition.

\begin{prop}\label{prop:infi-lift}
Let $Z$ be an elementary abelian $p$-group and consider a group extension
$$ 0 \to Z \to \calG \to \Gbar \to 0,$$
such that the induced action of $\Gbar$ on $Z$ factors through $\Gbar^\ad$ making it a $\FF_p[\Gbar^\ad]$-module.
Assume that $Z$ is an indecomposable $\FF_p[\Gbar^\ad]$-module occurring in $\ad(\rhobar)_\FF$
(see Lemma~\ref{lem:indec}).

Then there is a lift $\rho_Z:G \to \GL_2(\FF[X]/(X^2))$ of~$\rhobar$
such that $\Image(\rho_Z) \cong \calG$ and $Z \cong \Gamma_{\rho_Z}$ as $\FF_p[\Gbar^\ad]$-modules.

The group extension is split in all cases except possibly if $p=2$ and $Z=C(1)=\FF_2$. In that case,
there are two non-isomorphic extensions.
\end{prop}

\begin{proof}
In order to compute the possible group extensions, we first observe that since $p \nmid |\Hbar|$,
by inflation-restriction \cite[Proposition~1.6.7]{NSW} we obtain an isomorphism
$$\h^2(\Gbar,Z) \cong \h^2(\Gbar/\Hbar, Z^\Hbar).$$
For $p>2$, the latter group is always zero because $p \nmid  |\Gbar/\Hbar| =2$.
Thus the group extension in question is always split.

For $p=2$, we analyse the three possibilities for~$Z$ (see Lemma~\ref{lem:indec}) individually.
As $I(\chibar/\chibar^\sigma)^\Hbar = 0$, we find $\h^2(\Gbar,I(\chibar/\chibar^\sigma)) = 0$
and the corresponding group extension is split.
In order to compute the result for $N$, we make use of the fact that for a cyclic group $\h^2$ is isomorphic to
the $0$-th Tate (or modified) cohomology group (see e.g.\ \cite[\S1.2]{NSW}), which can be described explicitly. More precisely,
\begin{multline*}
\h^2(\Gbar,N) \cong \h^2(\Gbar/\Hbar,N) \cong \hhat^0(\Gbar/\Hbar,N)\\
= N^{\Gbar/\Hbar} / (1+\rhobar(\sigma))N 
= N^\Gbar / (1+\rhobar(\sigma))N \cong  \FF_2/\FF_2 = 0,
\end{multline*}
so that also the corresponding group extension is split.
With the same arguments, the case $Z=C(1)=\FF_2$ leads to
$$ \h^2(\Gbar,\FF_2) \cong \FF_2 / (1+\sigma)\FF_2 = \FF_2.$$
Consequently, there are two non-isomorphic group extensions of $\Gbar$ by~$\FF_2$.

When the sequence splits, we can proceed as follows. We have the exact sequence of groups
$$ \xymatrix@=.5cm{
0 \ar@{->}[r] &  M_2(\FF) \ar@{->}[r] & \GL_2(\FF[X]/(X^2)) \ar@{->}[r] & \GL_2(\FF) \ar@{->}[r] \ar@{->}@/^1pc/[l]^(.5){s} & 0 \\
& Z \ar@{^(->}[u] && \Gbar \ar@{^(->}[u]\\ 
}$$
and the action of $\Gbar^\ad$ on~$Z$ induced from this exact sequence is the action on $Z$ as a submodule of~$\ad(\rhobar)_\FF$.
We can thus simply obtain the split group extension of $\Gbar$ by~$Z$ as the subgroup of $\GL_2(\FF[X]/(X^2))$ generated
by $Z$ and $s(\Gbar)$.

If the sequence is non-split, then the analysis above implies that $p=2$ and $Z=\FF_2$. In order to treat this case, we make use of the case $p=2$, $Z=N$, 
where, as in Lemma~\ref{lem:dihedral}, we view $N = 1+X \{\mat a00d \in M_2(\FF) \;|\; a,d \in \FF_2\}$.
We now define the group $\calG' \subset \GL_2(\FF[X]/(X^2))$ as the group generated by $s(\Hbar)$, the
scalars $\FF_2 =  1+X\{\mat a00a \in M_2(\FF) \;|\; a \in \FF_2\} \subset N$ and the element $\mat 0 {1+X} {\chibar(\sigma^2)} 0$.
Let $n$ be the order of $\chibar(\sigma^2)$ in $\FF^{\times}$.
Reducing the matrices in~$\calG'$ modulo~$X$, we clearly obtain~$\Gbar$ and we have that the kernel of the reduction map  
is~$\FF_2$. Moreover, the element $\mat 0 {1+X} {\chibar(\sigma^2)} 0$ is a lift of~$\rhobar(\sigma)$, but it has order~$4n$, contrary
to the split case which does not contain any element of order~$4n$. This shows that $\calG'$ is a non-split extension of $\Gbar$ by $\FF_2$. Recall that our calculation shows that there are only $2$ isomorphism classes of extensions of $\Gbar$ by $\FF_2$. Since both $\calG$ and $\calG'$ are non-trivial extensions of $\Gbar$ by $\FF_2$, it follows that $\calG' \cong \calG$. 
\end{proof}

\begin{setup}\label{setup:three}
In the context of Set-ups \ref{setup:one} and~\ref{setup:two}, assume now also that $G$ satisfies Mazur's finiteness
condition $\Phi_p$ (see \cite[\S1.1]{M}).
\end{setup}

Since $\rhobar$ is irreducible, the deformation functor of $\rhobar$ for the category~$\calC$ is representable (see \cite[\S1.2]{M})).
One thus has a universal deformation of $\rhobar$
$$\rho^\univ: G \to \GL_2(R^\univ).$$
Write $\fm_\univ$ for $\fm_{R^\univ}$ and $\rho^\univ_\inft := (\rho^\univ)_\inft$, as well as
$\Gamma^\univ := \Gamma_{\rho^\univ}$ and $\Gamma_\inft^\univ := \Gamma_{\rho_\inft^\univ}$.

\begin{prop}\label{prop:univ-inft}
Let $\rho: G \to \GL_2(R)$ be a lift of~$\rhobar$. Then the morphism $R^\univ \to R$ existing by universality
induces a surjection $\Gamma^\univ_\inft \twoheadrightarrow \Gamma_\rho/\Phi(\Gamma_\rho)$.
\end{prop}

\begin{proof}
From the exact sequence $0 \to \Gamma_\rho \to \Image(\rho) \to \Gbar \to 0$
we obtain the group extension
\begin{equation}\label{eq:calG}
0 \to \Gamma_\rho/\Phi(\Gamma_\rho) \to \calG \to \Gbar \to 0.
\end{equation}

Let $V = \Gamma_\rho/\Phi(\Gamma_\rho)$ and decompose it into a direct sum of indecomposable $\FF_p[\Gbar^\ad]$-modules.
Corollary~\ref{cor:fratt-ad} allows us to apply Proposition~\ref{prop:infi-lift} to each of the indecomposable summands,
yielding that the group extension in~\eqref{eq:calG} can be realised by an infinitesimal deformation
$\rho_V:G \to \GL_2(T)$ for  $T=\FF[X_1,\dots,X_r]/(X_iX_j \;|\; i,j)$
for some~$r$, {\it i.e.} $\calG = \Image(\rho_V)$ and, in particular, $\Gamma_{\rho_V} \cong V$.

Let $\varphi: R^\univ \to T$ be the morphism existing by universality.
As $\fm_T^2 = 0$ and $p \fm_T = 0$, it follows that $\varphi$ factors over $(\fm_\univ^2,p)$ and thus
induces a surjection $\Image(\rho^\univ_\inft) \twoheadrightarrow \Image(\rho_V)$.
In particular, we obtain a surjection $\Gamma^\univ_\inft \twoheadrightarrow V$, as claimed. 
\end{proof}

We can now give the remaining proofs in the representation theory part of the paper.

\begin{proof}[Proof of Corollary~\ref{cor:univ-inft}.]
Proposition~\ref{prop:univ-inft} gives the surjection $\Gamma^\univ_\inft \twoheadrightarrow \Gamma^\univ/\Phi(\Gamma^\univ)$,
which has to be an isomorphism because $\Gamma^\univ_\inft$ is an elementary abelian $p$-quotient of~$\Gamma^{\univ}$ while
$\Gamma^\univ/\Phi(\Gamma^\univ)$ is the largest such. 
\end{proof}

\begin{lem}\label{lem:klein}
Suppose any deformation of~$\rhobar$~to~$\FF[X]/(X^2)$ is dihedral. Then there exists a subgroup $H \lhd G$ of index $2$ such that if $\rho : G \to \GL_2(\FF[X]/(X^2))$ is a deformation of $\rhobar$, then $\rho \simeq \Ind_H^G (\chi)$ for some character $\chi : H \to (\FF[X]/(X^2))^\times$.
\end{lem}

\begin{proof}
Suppose the projective image of $\rhobar$ is a non-abelian dihedral group.
Then there is a unique index $2$ subgroup $H \lhd G$ such that $\rhobar \simeq \Ind_H^G(\chibar)$ for some character $\chibar : H \to \FF^{\times}$.
Therefore, the lemma follows directly in this case.

Now suppose the projective image of $\rhobar$ is $\ZZ/2\ZZ \times \ZZ/2\ZZ$.
Then there are exactly $3$ index $2$ subgroups $H_1$, $H_2$ and $H_3$ of $G$ and characters $\chibar_i : H_i \to \FF^{\times}$ such that $\rhobar \simeq \Ind_{H_i}^G (\chibar_i)$ for every $1 \leq i \leq 3$.
Note that in this case $p$ is odd. 
Suppose $\rho : G \to \GL_2(\FF[X]/(X^2))$ is a deformation of $\rhobar$.
Then we can twist $\rho$ by a suitable character to get a deformation of $\rhobar$ with constant determinant.
Hence, it suffices to prove the lemma for deformations of $\rhobar$ with constant determinant.
Moreover, by part \eqref{lem:indec:b} of Lemma~\ref{lem:indec}, we get that as an $\FF_p[\Gbar^\ad]$-module, $\ad^0(\rhobar)_{\FF} \simeq \oplus_{i=1}^{3}C(\epsilon_i)_{\FF}$, where $\epsilon_i : G \twoheadrightarrow G/H_i \twoheadrightarrow \{\pm1\} \subset \FF^{\times}$ for $1 \leq i \leq 3$.

It follows from Theorem~\ref{thm:inf-dih} that any non-trivial deformation $\rho : G \to \GL_2(\FF[X]/(X^2))$ of $\rhobar$ with constant determinant can only be induced from a single subgroup of $G$ of index~$2$. Suppose not all such deformations are induced from the same subgroup. Then, without loss of generality for $i=1,2$, there are deformations $\rho_i : G \to \GL_2(\FF[X]/(X^2))$ of $\rhobar$ with constant determinant such that $\rho_i \simeq \Ind_{H_i}^G (\chi_i)$ for some character $\chi_i : H_i \to (\FF[X]/(X^2))^\times$.
Then for $i=1,2$, $\rho_i$ corresponds to a non-zero element $x_i \in H^1(G,C(\epsilon_i)_{\FF}) \subset H^1(G,\ad^0(\rhobar)_{\FF})$.
Let $\rho' : G \to \GL_2(\FF[X]/(X^2))$ be the deformation of $\rhobar$ corresponding to the element $x_1+x_2$ of $H^1(G,\ad^0(\rhobar)_{\FF})$.
Note that $x_1 + x_2 \in H^1(G, C(\epsilon_1)_{\FF} \oplus C(\epsilon_2)_{\FF}) \setminus \cup_{i=1}^{2} H^1(G, C(\epsilon_i)_{\FF})$.
Hence, the action of $\Hbar^\ad_i$ on $\Gamma_{\rho'}$ is not trivial for every $1 \leq i \leq 3$.
Therefore, Theorem~\ref{thm:inf-dih} implies that $\rho'$ is not dihedral.
This finishes the proof of the lemma. 
\end{proof}

\begin{proof}[Proof of Theorem~\ref{thm:inf-dih-univ}.]
(a) The implications `\eqref{item:idu:i} $\Rightarrow$ \eqref{item:idu:ii}' and `\eqref{item:idu:iii} $\Rightarrow$ \eqref{item:idu:iv}' are trivial and the implication 
`\eqref{item:idu:iv} $\Rightarrow$ \eqref{item:idu:i}' is immediate from Theorem~\ref{thm:inf-dih} and Corollary~\ref{cor:univ-inft}.
In order to see `\eqref{item:idu:ii} $\Rightarrow$\eqref{item:idu:iii}', consider any infinitesimal deformation~$\rho$ of~$\rhobar$.
Then, by Corollary~\ref{cor:fratt-ad}, the associated $\Gamma_\rho$ is an $\FF_p[\Gbar^\ad]$-module
the indecomposable submodules of which occur in $\ad(\rhobar)$.
By Proposition~\ref{prop:infi-lift}, each such indecomposable submodule $Z$ gives a representation of the type considered in~\eqref{item:idu:ii},
and is thus dihedral. 
By Lemma~\ref{lem:klein}, there exists an index $2$ subgroup $H \lhd G$ such that the representation given by each indecomposable submodule $Z$ of $\Gamma_\rho$ is induced from a character of $H$.
By Theorem~\ref{thm:inf-dih} this means that any such $Z$ is trivial as $\FF_p[\Hbar^\ad]$-module.
Thus $\Gamma_\rho$ is trivial as $\FF_p[\Hbar^\ad]$-module, whence $\rho$ is dihedral by Theorem~\ref{thm:inf-dih}.

(b) Let $R_\chibar^\univ$ be the universal deformation ring of~$\chibar$ as discussed in Proposition~\ref{prop:univ-char}.
Now we are assuming that $\rho^\univ=\Ind_H^G(\chi)_{R^\univ}$ for some character~$\chi$. Since $\rho^\univ$ is a deformation of~$\rhobar$,
the character $\chi$ is a deformation of~$\chibar$ (if, by restriction to~$H$, we find that $\chi^\sigma$ deforms $\chibar$,
then we simply replace $\chi$ by $\chi^\sigma$), giving a morphism
$\alpha: R_\chibar^\univ \to R^\univ$. On the other hand, given the deformation $\chi^\univ$ of~$\chibar$,
we obtain a deformation $\Ind_H^G(\chi^\univ)_{R_\chibar^\univ}$ of~$\rhobar$ and thus a morphism $\beta: R^\univ \to R_\chibar^\univ$.
For the composite we have $\alpha \circ \beta \circ \rho^\univ=\rho^\univ$ and hence $\alpha \circ \beta$ is the identity.
Similarly, $\beta \circ \alpha \circ \chi^\univ = \chi^\univ$, whence $\beta \circ \alpha$ is the identity,
implying that both $\alpha$ and $\beta$ are isomorphisms and $R^\univ \cong R^\univ_\chibar$, as claimed. 
\end{proof}

\begin{rem}
Let $R = \FF[\epsilon]/(\epsilon^3)$ for a prime $p>2$ and consider the $p$-group
$$\Gamma = \{1 + \epsilon \mat r00{-r} +\epsilon^2 \mat abc{r^2-a} \;|\; a,b,c,r \in \FF_p\} \subset \GL_2(R).$$
It is stable under conjugation by matrices of the form $\mat *00*$ and $\mat 0**0$, i.e.\ by~$\Gbar$ viewed inside~$\GL_2(R)$.
Moreover, $\Gamma$ is an elementary abelian $p$-group, so that $\Phi(\Gamma) = 0$ and $\Gamma$ is its own Frattini quotient.
Let $\calG \subset \GL_2(R)$ be the subgroup generated by~$\Gbar$ and~$\Gamma$.
Then any lift $\rho: G \to \GL_2(R)$ of $\rhobar$ with image~$\calG$ (and then also $\Gamma=\Gamma_\rho$)
provides an example where $\rho_\inft$ is dihedral but $\rho$ is not (in view of Theorem~\ref{thm:inf-dih}).
\end{rem}

\subsection*{Appendix: An alternative proof}

In this appendix, we prove a generalisation of Theorem~\ref{thm:inf-dih-univ} using a different argument.

To be precise, let $G$ be a profinite group satisfying the finiteness condition $\Phi_p$ of Mazur (see \cite[Section $1.1$]{M}). Let $p$ be a prime, $E$ be a finite extension of $\mathbb{Q}_p$. Let $\rhobar : G \to GL_2(E)$ be an absolutely irreducible dihedral representation. So there exists an open normal subgroup $H$ of $G$ of index $2$ and a character $\chibar : H \to E^\times$ such that $\rho \simeq \Ind_H^G(\chibar)$. Note that $\rhobar|_H = \chibar \oplus \chibar^{\sigma}$.

Let $\mathcal{D}$ be the category of local Artinian algebras with residue field $E$. Then, it follows, from \cite[Lemma $9.3$]{Ki}, that the functor from $\mathcal{D}$ to the category of sets sending an object $R$ of $\mathcal{D}$ to the set of deformations of $\rhobar$ to $\GL_2(R)$ is pro-representable by a local complete Noetherian ring with residue field $E$. Let $R^{\univ}$ be the universal deformation ring of $\rhobar$, $\fm^{\univ}$ be the maximal ideal of $R^{\univ}$ and $\rho^{\univ} : G \to GL_2(R^{\univ})$ be the universal deformation of $\rhobar$. 

\begin{thm}\label{thm:altpf}
Suppose $p$ is a prime, $E$ is a finite extension of $\QQ_p$ or $\FF_p$ and $G$ is  a profinite group satisfying Mazur's finiteness condition $\Phi_p$. Let $\rhobar : G \to GL_2(E)$ be an absolutely irreducible dihedral representation. Then $\rho^{\univ}$ is dihedral if and only if any infinitesimal deformation $\tilde\rho : G \to \GL_2(E[\epsilon]/(\epsilon^2))$ of $\rhobar$ is dihedral.
\end{thm}

\begin{proof}
The forward direction is trivial. So, we only need to prove the backward direction. Let $H \lhd G$ be an index $2$ subgroup such that $\rhobar \simeq \Ind_H^G \chibar$ for some character $\chibar : H \to E^\times$. Now the proof of Lemma~\ref{lem:klein} implies that to prove the theorem, it suffices to prove that if $\rho^\univ$ is not dihedral, then for every such $H$, there exists an infinitesimal deformation of $\rhobar$ which is not induced from a character of $H$. Note that a non-zero element of $H^1(G,\Ind_H^G(\chibar/\chibar^{\sigma}))$ gives us an infinitesimal deformation of $\rhobar$ which is not induced from a character of $H$.
Thus it suffices to prove that if $\rho^\univ$ is not dihedral, then $H^1(G,\Ind_H^G(\chibar/\chibar^{\sigma})) \neq 0$.
By Shapiro's lemma, we have $H^1(G,\Ind_H^G(\chibar/\chibar^{\sigma})) \simeq H^1(H,\chibar/\chibar^{\sigma})$.
So, it suffices to prove $H^1(H,\chibar/\chibar^{\sigma}) \neq 0$.

We will now assume $\rho^\univ$ is not dihedral and construct a non-zero element of $H^1(H,\chibar/\chibar^{\sigma})$. It follows, from part (b) of Lemma~\ref{lem:split-basis}, that there exists an element $g_0 \in H$ such that $\rho^{\univ}(g_0) =\begin{pmatrix} a & 0\\ 0 & b\end{pmatrix}$ with $a \neq b$.
Note that $M_2(R^{\univ})$ is a Generalized Matrix Algebra (GMA) (see \cite[Chapter $1$]{BC} for the definition). Therefore, by \cite[Lemma $2.4.5$]{Bel}, we get that $A = R^{\univ}[\rho^{\univ}(H)]$ is a sub-$R^{\univ}$-GMA of $M_2(R^{\univ})$. 

Recall that $A$ being a sub-$R^{\univ}$-GMA of $M_2(R^{\univ})$ means that $A = \begin{pmatrix} R^{\univ} & B\\ C & R^{\univ}\end{pmatrix}$, where $B$ and $C$ are ideals of $R^{\univ}$ (see \cite[Section $2.2$]{Bel}). As $R^{\univ}$ is Noetherian, $B$ and $C$ are finitely generated $R^{\univ}$-modules. Moreover, since the image of $A$ modulo $\fm^{\univ}$ is diagonal, it follows that $B \subset \fm^{\univ}$ and $C \subset \fm^{\univ}$. 

As we have assumed that $\rho^\univ$ is not dihedral and $H$ is normal in $G$, it follows that both $B \neq 0$ and $C \neq 0$. Indeed, non-dihedralness implies that at least one of them is not $0$ and if the other one is $0$, then normality of $H$ in $G$ implies that $\rho^\univ$ is reducible, which is not possible.

Now suppose for $h \in H$, $\rho^{\univ}(h)=\begin{pmatrix} a(h) & b(h)\\c(h) & d(h)\end{pmatrix}$. So $b(h) \in B$ and $c(h) \in C$. Let $\phi : B \to E$ be a non-trivial map of $R^\univ$-modules. This defines a representation $\rho_{\phi} : H \to \GL_2(E)$ such that $\rho_{\phi}(h)=\begin{pmatrix} \chibar(h) & \phi(b(h))\\ 0 & \chibar^{\sigma}(h)\end{pmatrix}$. So, this defines an element of $H^1(H,\chibar/\chibar^{\sigma})$. Since the $R^\univ$-span of $\rho^{\univ}(H)$ is $\begin{pmatrix} R^\univ & B\\ C & R^\univ\end{pmatrix}$, it follows that $E$-span of $\rho_{\phi}(H)$ contains all the diagonal matrices and $\begin{pmatrix} 0 & 1\\ 0 & 0\end{pmatrix}$. As $\begin{pmatrix} 0 & 1\\ 0 & 0\end{pmatrix}$ does not commute with all diagonal matrices, it follows that $\rho_{\phi}(H)$ is not abelian and hence, it defines a non-zero element of $H^1(H,\chibar/\chibar^{\sigma})$. This finishes the proof of the theorem.
\end{proof}

\section{Number theoretic dihedral universal deformations}\label{sec:nt}

In this section, we turn our attention to dihedral Galois representations of number fields and their deformations and develop and
prove the results outlined in section~\ref{sec:intro-nt}.
We keep the notation introduced previously. In addition, we use the following notation.

\begin{notation}\label{not:nt}
For a number field $N$, denote by $G_N$ the absolute Galois group of $N$ and by $A(N)$ the class group of $N$. If $\mathfrak{p}$ is a prime of $N$, then denote by $N_{\mathfrak{p}}$ the completion of $N$ at $\mathfrak{p}$ and by $G_{N_{\mathfrak{p}}}$ the decomposition group of $\mathfrak{p}$ inside $G_N$. If $\rho$ is a representation of a group $G$ and $H$ is a subgroup of $G$, then we denote by $\rho|_H$ the restriction of $\rho$ to $H$.
If $L$ and $K$ are two fields such that $L$ is an algebraic Galois (but not necessarily finite) extension of $K$,
then we denote the Galois group $\Gal(L/K)$ by $G_{L/K}$.
Let $\mu_p$ be the group of $p$-th roots of unity inside an algebraic closure of the prime field.

For an extension $N/K$ of number fields and a set of places $S$ of~$K$,
denote by $N(S)$ the maximal extension of $N$ unramified outside the primes of $N$ lying above~$S$.
Note that for a Galois extension $N/K$, the extension $N(S)/K$ is also Galois as any conjugate $\sigma(N(S))$ for $\sigma$ fixing~$K$
is also unramified over~$N$ outside the primes of $N$ lying above~$S$. 
Furthermore, let $N(S)^{\ab,p}$ be the maximal abelian extension of $N$ inside $N(S)$ of exponent~$p$.
\end{notation}

\begin{setup}\label{setup:dih}
Let $K$ be a number field, $L$ be a quadratic extension of $K$ and $\chibar : G_L \to \FF^\times$ be a character such that the representation $\rhobar = \Ind_{G_L}^{G_K} (\chibar) : G_K \to \GL_2(\FF)$ is absolutely irreducible. So, $\rhobar|_{G_L} = C(\chibar)_\FF \oplus C(\chibar^{\sigma})_\FF$ where $\chibar^{\sigma}(h)=\chibar(\sigma h \sigma^{-1})$ in the notation of Definition~\ref{defi:modules}.
Let $M^{\rhobar}$ be the extension of $K$ fixed by $\ker(\rhobar)$ and $M^{\ad}$ be the extension of $K$ fixed by $\ker(\ad(\rhobar))$. 
If $p=2$, assume that $M^\ad$ is totally imaginary.

Let $\Gbar^\ad = \Gal(M^{\ad}/K)$ and $\Hbar^\ad = \Gal(M^{\ad}/L)$. So, $\Hbar^\ad$ is a cyclic subgroup of index 2 in $\Gbar^\ad$.
Also let $C := \ker(\Gbar \twoheadrightarrow \Gbar^\ad)$.
Let $S_{\infty}$ be the set of all archimedean places of $K$ (places of $K$ lying above $\infty$).
Let $S_p$ be the set of primes of~$K$ lying above~$p$.
Furthermore, let $S_{\rhobar}$ be the finite set of finite primes of~$K$ at which $M^{\rhobar}$ is ramified over~$K$.
Let $S$ be a finite set of primes of~$K$ such that $S_\infty \subseteq S$ and $S \cap S_p = \emptyset$.
Let $\kappa = K(S \cup S_\rhobar)$.
\end{setup}

We summarise some of the fields and Galois groups in the following diagram (some notation in the diagram is only introduced later).
$$ 
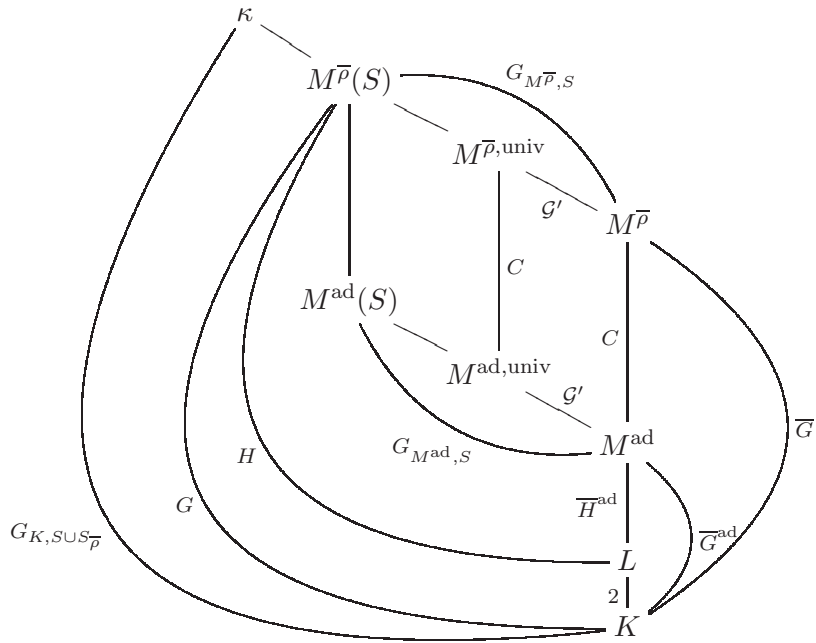
\xymatrix@=.4cm{
\kappa \ar@{-}@/_12pc/[dddddddddrrr]_{G_{K,S\cup S_\rhobar}} \ar@{-}[dr] \\
& M^{\rhobar}(S) \ar@{-}[dr]\ar@{-}@/^2pc/[ddrr]^{G_{M^{\rhobar},S}} \ar@{-}[ddd] \ar@{-}@/_10pc/[rrdddddddd]_G \ar@{-}@/_8pc/[rrddddddd]_H\\
&& M^{\rhobar,\univ} \ar@{-}[ddd]^C  \ar@{-}[dr]_{\calG'}\\
& & & M^{\rhobar}\ar@{-}[ddd]_C \ar@{-}@/^5pc/[dddddd]^\Gbar  \\
& M^\ad(S) \ar@{-}[dr]\ar@{-}@/_2pc/[ddrr]_{G_{M^\ad,S}} \\
&& M^{\ad,\univ}\ar@{-}[dr]^{\calG'}\\
& & & M^\ad \ar@{-}[dd]_{\Hbar^\ad} \ar@{-}@/^2pc/[ddd]^{\Gbar^\ad}\\
\\
& & & L\ar@{-}[d]_2\\
& & & K\\
}$$
Note that the extensions $M^{\rhobar}(S)$ and $M^{\ad}(S)$ of $K$ are Galois. Put $G = \Gal(M^{\rhobar}(S)/K)$ and $H = \Gal(M^{\rhobar}(S)/L)$.
Note that these pieces of notation exactly correspond to those of section~\ref{sec:rep}.

Let $G_{M^{\ad},S} = \Gal(M^{\ad}(S)/M^{\ad})$ and $G^{\ab}_{M^{\ad},S}$ be the continuous abelianisation of $G_{M^{\ad},S}$. As $\Gal(M^{\ad}(S)/M^{\ad})$ is normal in $\Gal(M^{\ad}(S)/K)$, the closure of the commutator subgroup of $\Gal(M^{\ad}(S)/M^{\ad})$ is also normal in $\Gal(M^{\ad}(S)/K)$. So, we get an action of $\Gbar^\ad = \Gal(M^{\ad}/K)$ on $G^{\ab}_{M^{\ad},S}$ and hence, on the $\FF_p$-vector space $\calG := G^{\ab}_{M^{\ad},S}/(G^{\ab}_{M^{\ad},S})^p$  by conjugation.

Let $S''$ be the subset of~$S$ consisting of the finite primes $q$ such that $\mu_p \subseteq M^\ad_{q'}$ for some (and then every) prime $q'$ of~$M^\ad$ dividing~$q$
(note: $S=S'' \cup S_\infty$ if $p=2$).
Denote by $D_q$ a decomposition group of~$q$ inside~$\Gbar^\ad$.
Let $\chi_p^{(q)}$ be the modulo~$p$ cyclotomic character viewed as a character of~$D_q$ for $q \in S''$ (note $\chi_2^{(q)}$ is the trivial character).

\begin{prop}\label{prop:CL-Ind}
The elementary abelian $p$-group $\calG$ admits $A(M^\ad)/pA(M^\ad)$ as a quotient
and $\calM = \ker\big(\calG \twoheadrightarrow A(M^\ad)/pA(M^\ad)\big)$ is isomorphic to a quotient of
$\prod_{q \in S''} \Ind_{D_q}^{\Gbar^\ad}(C(\chi_p^{(q)})_{\FF_p})$ as $\FF_p[\Gbar^\ad]$-modules.
\end{prop}

\begin{proof}
See also \cite[Section $1.2$]{BM}. Let $Y$ be the Galois group of $M^\ad(S)^\ab/M^\ad$, the maximal abelian extension of~$M^\ad$ unramified outside
the primes above~$S$. Note $\calG = Y/Y^p$. By global class field theory, we have the exact sequence of $\Gbar^\ad$-modules: 
\begin{equation}\label{eq:cft1}
\prod_{q' \in S'_\fin} \calO^\times_{q'} \times \prod_{v \in S'_\real}\ZZ/2\ZZ \to Y \to A(M^\ad) \to 0,
\end{equation}
where $S'_\fin$ is the set of finite primes of $M^\ad$ lying above~$S$, $\calO_{q'}$ is the ring of integers in $M^\ad_{q'}$ and $S'_\real$ is the set consisting of all real places of~$M^\ad$. Recall that we have assumed $M^{\ad}$ to be totally complex if $p=2$. Hence, taking the exact sequence~\eqref{eq:cft1} modulo~$p$, we obtain the exact sequence of $\FF_p[\Gbar^\ad]$-modules: 
\begin{equation}
\prod_{q' \in S'_\fin} \calO^\times_{q'}/(\calO^\times_{q'})^p \to Y/Y^p \to A(M^\ad)/pA(M^\ad) \to 0.
\end{equation}
As $\calO^\times_{q'}$ is the direct product of the group of roots of unity in~$M^\ad_{q'}$ and the group of $1$-units (which is a pro-$q$ group),
it follows that $\calO^\times_{q'}/(\calO^\times_{q'})^p$ is non-trivial if and only if $\mu_p \subseteq M^\ad_{q'}$.
In that case $\calO^\times_{q'}/(\calO^\times_{q'})^p$ is isomorphic to $C(\chi_p^{(q'/q)})_{\FF_p}$ as $\FF_p[D_{q'/q}]$-modules,
where, for a moment, we keep track of the prime $q'$ above~$q$ by denoting the decomposition group inside~$\Gbar^\ad$ corresponding to the prime~$q'$
by $D_{q'/q}$ and writing $\chi_p^{(q'/q)}$ for its modulo~$p$ cyclotomic character.

Thus $\calM$ is an $\FF_p[\Gbar^\ad]$-quotient of $\prod_{q \in S''} \prod_{q' \mid q} C(\chi_p^{(q'/q)})_{\FF_p}$.
For a fixed $q \in S''$, $\Gbar^\ad$ permutes the ideals $q'\mid q$ and one obtains that, as $\FF_p[\Gbar^\ad]$-modules,
$\prod_{q \in S''} \prod_{q' \mid q}  C(\chi_p^{(q'/q)})_{\FF_p}$ is isomorphic to $\prod_{q \in S''}\Ind_{D_q}^{\Gbar^\ad}(C(\chi_p^{(q)})_{\FF_p})$,
which does not depend on the choice of $q'$ above~$q$ (whence we simplified notation). 
\end{proof}

For further analysis, we first define the following sets of finite primes of $K$:
\begin{enumerate}
\item Let $S_1$ be the set of primes $\ell$ of $K$ not lying above $p$ such that $\ell$ is split in $L$ and $\chibar/\chibar^{\sigma}|_{D_\ell}$ is either $\chi_p^{(\ell)}$ or $(\chi_p^{(\ell)})^{-1}$. Note that this implies that $\ell$ is unramified in $M^{\ad}$ and for any prime $\lambda$ of $M^{\ad}$ lying above $\ell$, $M^{\ad}_{\lambda}=K_{\ell}(\mu_p)$. 
\item Let $S_2$ be the set of primes $\ell$ of $K$ not lying above $p$ such that $\ell$ is not split in $L$, for any prime $\lambda$ of $M^{\ad}$ lying above $\ell$, $\mu_p \subseteq M^{\ad}_{\lambda}$ and $[M^{\ad}_{\lambda} : K_{\ell}]=2$ (so, the unique prime of $L$ lying above $\ell$ splits completely in $M^{\ad}$). 
\item Let $S_3$  be the set of primes $\ell$ of $K$ not lying above $p$ such that $[K_{\ell}(\mu_p) : K_{\ell}] = 2$, $\ell$ is ramified in $L$ and, for any prime $\lambda$ of $M^{\ad}$ lying above $\ell$, $\Gal(M^{\ad}_{\lambda}/K_{\ell}) \simeq \ZZ/2\ZZ \times \ZZ/2\ZZ$ and $\mu_p \subseteq M^{\ad}_{\lambda}$ (so, the unique prime of $L$ above $\ell$ is unramified in $M^{\ad}$).
\end{enumerate}
Let $S_0 = S_1 \cup S_2 \cup S_3$.

As $M^{\ad}$ is Galois over $K$, if $\mu_p \subseteq M^{\ad}_{\lambda_0}$ for some prime $\lambda_0$ of $M^{\ad}$ lying above $\ell$, then $\mu_p \subseteq M^{\ad}_{\lambda}$ for all primes $\lambda$ of $M^{\ad}$ lying above $\ell$. Moreover, $S_3 = \emptyset$ when $\chibar/\chibar^{\sigma}$ is of odd order. Observe that, when $\mu_p \subseteq K$ (thus, in particular, when $p=2$), we have:
\begin{enumerate}
\item $S_1$ is the set of primes of $K$ not lying above $p$ which are completely split in $M^{\ad}$. 
\item $\ell \in S_2$ if and only if $\ell$ is not a prime above $p$, $\ell$ is either inert or ramified in $L$ and the unique prime of $L$ lying above $\ell$ is completely split in $M^{\ad}$.
\item $S_3 = \emptyset$.
\end{enumerate}

\begin{prop}\label{dimprop}
If $S \cap S_0 = \emptyset$ and $\Hom_{\FF_p[\Gbar^\ad]}(A(M^{\ad})/pA(M^{\ad}) , I(\chibar/\chibar^{\sigma})) = 0$ in the notation of Definition~\ref{defi:modules},
then $\Hom_{\FF_p[\Gbar^\ad]}(\calG,  I(\chibar/\chibar^{\sigma})) = 0$.
\end{prop}

\begin{proof}
By Proposition~\ref{prop:CL-Ind}, the notation of which we continue to use, restriction gives an injection
\begin{multline}\label{eq:hom1}
\Hom_{\FF_p[\Gbar^\ad]} (\calG  , I(\chibar/\chibar^{\sigma})) 
 \hookrightarrow \Hom_{\FF_p[\Gbar^\ad]}( \prod_{q \in S''} \Ind_{D_q}^{\Gbar^\ad}(C(\chi_p^{(q)})_{\FF_p}), I(\chibar/\chibar^{\sigma})) \\
 \cong \prod_{q \in S''} \Hom_{\FF_p[\Gbar^\ad]}( \Ind_{D_q}^{\Gbar^\ad}(C(\chi_p^{(q)})_{\FF_p}), I(\chibar/\chibar^{\sigma})).
\end{multline}
Frobenius reciprocity (\cite[Thm.~10.8]{CR}) yields
\begin{equation}\label{eq:frob-rec}
\Hom_{\FF_p[\Gbar^\ad]}( \Ind_{D_q}^{\Gbar^\ad} (C(\chi_p^{(q)})_{\FF_p}), I(\chibar/\chibar^{\sigma})) = \Hom_{\FF_p[D_q]}(C(\chi_p^{(q)})_{\FF_p}, I(\chibar/\chibar^{\sigma})|_{D_q}).
\end{equation}
We see that $\Hom_{\FF_p[D_q]}( C(\chi_p^{(q)})_{\FF_p}  , I(\chibar/\chibar^{\sigma})|_{D_q}) \neq 0$ if and only if
$ C(\chi_p^{(q)})_{\FF_p}$ is an $\FF_p[D_q]$-submodule of $I(\chibar/\chibar^{\sigma})|_{D_q}$.
Now, we will do a case-by-case analysis of when this will happen.
Let us point out that, in general, $I(\chibar/\chibar^{\sigma})|_{D_q}$ is defined over some extension $\FF/\FF_p$.
Note that $C(\chi_p^{(q)})_{\FF_p}$ is an $\FF_p[D_q]$-submodule of some module $C(\psi)_\FF$ (for some $\FF^\times$-valued character $\psi$ of~$D_q$)
if and only if $\chi_p^{(q)} = \psi$.
Note also that $\Gbar^\ad$ acts faithfully on $\ad(\rho)_\FF \cong N_\FF \oplus I(\chibar/\chibar^\sigma)_\FF$ (see Lemma~\ref{lem:dihedral})
and thus $\Hbar^\ad$ acts faithfully on $C({\chibar/\chibar^\sigma})_\FF$.

\begin{enumerate}
\item \textbf{$q$ is split in $L$:} In this case, $D_q \subseteq \Hbar^\ad$, whence
$I(\chibar/\chibar^{\sigma})_\FF|_{D_q} = C({\chibar/\chibar^\sigma}|_{D_q})_\FF \oplus C({\chibar^\sigma/\chibar}|_{D_q})_\FF$.
So, $C(\chi_p^{(q)})_{\FF_p}$ is an $\FF_p[D_q]$-submodule of $I(\chibar/\chibar^{\sigma})|_{D_q}$ if and only if
$\chibar/\chibar^\sigma|_{D_q} = \chi_p^{(q)}$ or $\chibar^\sigma/\chibar|_{D_q} = \chi_p^{(q)}$.
Note that such primes~$q$ are exactly the ones lying in~$S_1$.

\item \textbf{$q$ is not split in $L$:} In this case, $I(\chibar/\chibar^{\sigma})|_{D_q}$ is reducible if and only if $D_q$ is either $\ZZ/2\ZZ$ or $\ZZ/2\ZZ \times \ZZ/2\ZZ$.
Let $\tilde{q}$ be the unique prime of $L$ lying above~$q$.

Suppose first that $D_q = \ZZ/2\ZZ$.
Then $D_q = \Gal(L_{\tilde q}/K_q)$. Consequently, $I(\chibar/\chibar^{\sigma})_\FF|_{D_q}$ is $\FF^2$ with
$D_q$-action swapping the two standard basis vectors.
So, if $p>2$, then $I(\chibar/\chibar^{\sigma})_\FF|_{D_q} \cong C(1)_\FF \oplus C(\epsilon)_\FF$ for the quadratic character
$\epsilon: D_q \cong \{\pm 1\} \subseteq \FF^\times$.
If $p=2$, then $I(\chibar/\chibar^{\sigma})|_{D_q}$ is the module~$N_\FF$ from Definition~\ref{defi:modules}.
Hence, for any $p$, we see that $C(\chi_p^{(q)})_{\FF_p}$ is an $\FF_p[D_q]$-submodule of $I(\chibar/\chibar^{\sigma})|_{D_q}$
if and only if $\chi_p^{(q)}$ is the trivial character or equal to~$\epsilon$. This happens if and only if one of the following conditions hold:
\begin{enumerate}
\item $\mu_p \subseteq K_q$,
\item $L_{\tilde q} = K_q(\mu_p)$ (and then $q$ is inert in $L$).
\end{enumerate}
Now, the primes~$q$ satisfying any one of the conditions above are exactly the ones lying in $S_2$.

Suppose now that $D_q = \ZZ/2\ZZ \times \ZZ/2\ZZ$ (note that this case cannot happen when $p=2$).
Then $q$ is ramified in $M^{\ad}$ and is not split in~$L$.
Note that $\chibar/\chibar^{\sigma}$ is a non-trivial character of $\Gal(M^{\ad}_{q'}/L_{\tilde{q}}) = D_q \cap \Gal(M^{\ad}/L)$.
So, in this case, $C(\chi_p^{(q)})_{\FF_p}$ is an $\FF_p[\Gbar^\ad]$-submodule of $I(\chibar/\chibar^{\sigma})|_{D_q}$ if and only if
$M^\ad_{q'} = L_{\tilde{q}}(\mu_p) \overset{2}{\supsetneq} L_{\tilde{q}}\overset{2}{\supsetneq} K_q$.
This is equivalent to $[K_q(\mu_p) : K_q] =2$, $q$ is ramified in $L$ and the unique prime of $L$ lying above $q$ is unramified in $M^{\ad}$.
So, primes satisfying these conditions are exactly the ones belonging to $S_3$.
\end{enumerate}
Thus, the primes satisfying these conditions are contained in $S_1 \cup S_2 \cup S_3 = S_0$, whence by our assumption none of them lies in~$S$.
We thus obtain  $\Hom_{\FF_p[D_q]}(C(\chi_p^{(q)})_{\FF_p}, I(\chibar/\chibar^{\sigma})|_{D_q})= 0$ for all~$q \in S''$.
In view of \eqref{eq:hom1} and~\eqref{eq:frob-rec}, the assertion of the proposition follows. 
\end{proof}

\begin{rem}\label{ramrem}
In the proof of Proposition~\ref{prop:CL-Ind}, we are ignoring the contribution of the kernel of the first map of~\eqref{eq:cft1}, which is given by the global units $\calO^\times_{M^{\ad}}$ of $M^{\ad}$. So, it could happen that $S$ contains some primes from $S_0$ and the conclusion of the proposition still continues to hold due to the contribution coming from $\calO^\times_{M^{\ad}}$ negating the contribution coming from primes of $M^{\ad}$ lying above primes from~$S_0$.
However, we know the $\FF_p[\Gbar^\ad]$-module structure of the finite dimensional $\FF_p$-vector space $\calO^\times_{M^{\ad}}/(\calO^\times_{M^{\ad}})^p$. So, we can find a number $n_0$ such that if $S$ contains more than $n_0$ primes from $S_0$, then the statement of the proposition would not be true any more.
\end{rem}

\begin{rem}\label{classgrprem}
The assumption $\Hom_{\FF_p[\Gbar^\ad]}(A(M^{\ad})/pA(M^{\ad}) ,  I(\chibar/\chibar^{\sigma})) = 0$ is necessary for Pro\-position~\ref{dimprop} to hold
because if the assumption is violated, then by \eqref{eq:cft1}, $\Hom_{\FF_p[\Gbar^\ad]}(\calG, I(\chibar/\chibar^{\sigma}))$ is non-zero.
\end{rem}

\begin{rem}\label{prem}
If we include all primes $\mathfrak{p}$ of $K$ lying above $p$ in $S$ and if either $K$ is not totally real or $\rhobar$ is not totally even (equivalently, if $M^{\ad}$ is not totally real), then either the existence of complex embeddings of $K$ or the oddness of $\rhobar$ would imply that the multiplicity of $I(\chibar/\chibar^{\sigma})$ occurring in $\prod_{\mathfrak{p'} | p} \calO_{\mathfrak{p'}}/(\calO_{\mathfrak{p'}})^p$ would be greater than the one of $I(\chibar/\chibar^{\sigma})$ occurring in $\calO^\times_{M^{\ad}}/(\calO^\times_{M^{\ad}})^p$ (here, $\calO_{\mathfrak{p'}}$ is the ring of integers of the completion of $M^{\ad}$ at the prime $\mathfrak{p'}$ lying above $\mathfrak{p}$). Thus, we see that in this case $\Hom_{\FF_p[\Gbar^\ad]}(\calG, I(\chibar/\chibar^{\sigma})) \neq 0$. Therefore, it is necessary that $S$ does not contain all the primes above $p$ for Proposition~\ref{dimprop} to hold when either $K$ is not totally real or $\rhobar$ is not totally even. From this logic, we also see that, if either $K$ is not totally real or $\rhobar$ is not totally even, then, in some cases, the presence of only some (and not all) of the primes of $K$ lying above $p$ in $S$ is sufficient to conclude that Proposition~\ref{dimprop} does not hold.
\end{rem}

\begin{rem}\label{rem:tot-im}
We have assumed that $M^{\ad}$ is totally complex when $p=2$. If $M^{\ad}$ had a real place~$v$, then its $\Gbar^\ad$-orbit would consist entirely of real places. So, there would be a contribution from these places in the exact sequence~\eqref{eq:cft1}.
Moreover, as $\FF_2[\Gbar^\ad]$-module, the contribution $\prod_{g \in \Gbar^\ad}\ZZ/2\ZZ$ given by the Galois orbit of $v$ in the first term of the exact sequence would be isomorphic to $\FF_2[\Gbar^\ad]$, i.e.\ the regular representation.
Hence, there can be a non-zero $\FF_2[\Gbar^\ad]$-homomorphism from the first term in \eqref{eq:cft1} to $I(\chibar/\chibar^{\sigma})$.
Thus, the hypothesis seems essential for the proposition unless we know that contribution from global units negates the contribution coming from $S_{\infty}$. So, in particular, if $M^{\ad}$ has sufficiently many real places, then the statement of Proposition~\ref{dimprop} does not hold.
\end{rem}

We now turn towards deformation theory.
Clearly, $\rhobar$ is a representation of $G = \Gal(M^{\rhobar}(S)/K)$. 
Denote by $D_S$ the functor from $\calC$ to the category of sets which sends $R$ to the set of continuous deformations $\rho : G \to \GL_2(R)$ of $\rhobar$. Let $D^0_S$ be the subfunctor of $D_S$ which sends an object $R$ of $\calC$ to the set of continuous deformations $\rho : G \to \GL_2(R)$ of $\rhobar$ with determinant $\widehat{\det(\rhobar)}$.

\begin{lem}\label{deflem}
The functors $D_S$ and $D^0_S$ are representable by rings in~$\calC$.
\end{lem}

\begin{proof}
The group $G$ is a quotient of $\Gal(\kappa/K)$ and, hence, satisfies the finiteness condition $\Phi_p$ of Mazur (\cite[1.1]{M}).
Therefore, as seen just before Proposition~\ref{prop:univ-inft}, $D_S$ is representable by a ring in~$\calC$.
As a consequence, it follows that $D^0_S$ is also representable by a ring in~$\calC$ (see \cite[Section $24$]{Ma}). 
\end{proof}

For an object $R$ of $\calC$, a deformation $\rho : \Gal(\kappa/K) \to \GL_2(R)$ of $\rhobar$ belongs to $D_S(R)$ if and only if the field fixed by $\ker(\rho)$ is an extension of $M^{\rhobar}$ unramified outside the places of $M^{\rhobar}$ lying above~$S$. Hence, we make the following definition.

\begin{defi}
Deformations of $\rhobar$ belonging to $D_S(R)$ ($D^0_S(R)$) are said to be {\em relatively unramified outside~$S$} (with constant determinant).
\end{defi}
Note that if $\rho \in D_S(R)$ and $\ell$ is a prime of $K$ not contained in $S$, then $\rho|_{G_{K_\ell}}$ is often called a minimal deformation of $\rhobar|_{G_{K_\ell}}$ in the literature (see \cite[Section 29]{Ma}). Thus an element of $D_S(R)$ is a deformation of $\rho$ which is minimally ramified at all primes not contained in $S$.

\begin{rem}
We are careful to speak of {\em relatively} unramified deformations of $\rhobar$ instead of just unramified ones in order to avoid possible confusion with unramified representations: if $S$ does not contain all of $S_\rhobar$, then a deformation can be relatively unramified outside~$S$ even though as a representation it does ramify outside~$S$.
\end{rem}

We continue to essentially follow the notation introduced in section~\ref{sec:rep} with the exception that we keep track of the chosen set of primes~$S$.
Denote by $R^{\univ}_S$ the ring representing $D_S$. Denote by $\rho^{\univ}_S : G \to \GL_2(R^{\univ}_S)$ the universal deformation of $\rhobar$ relatively unramified outside~$S$. Let $(R^{\univ}_S)^0$ be the ring which represents $D^0_S$ and $(\rho^{\univ}_S)^0 : G \to \GL_2((R^{\univ}_S)^0)$ the universal deformation of $\rhobar$ relatively unramified outside $S$ with constant determinant. So, we have a natural surjective homomorphism $R^{\univ}_S \to (R^{\univ}_S)^0$.
Let $\fm_{R^{\univ}_S}$ be the maximal ideal of $R^{\univ}_S$.

\begin{proof}[Proof of Theorem~\ref{thm:nt-dih}.]
Write $\rho := \rho^\univ_S$ as abbreviation. We shall apply some of the main results from section~\ref{sec:rep}.
In particular, it will suffice to work with $p$-Frattini quotients and to apply the classifications of modules from section~\ref{sec:rep}.

Let $\Gamma_\rho$ be the group defined in~\eqref{eq:seq} and let $\calG' := \Gamma_\rho/\Phi(\Gamma_\rho)$ be its $p$-Frattini quotient.
Let $M^{\rhobar,\univ}$ be the subfield of $M^\rhobar(S)$ such that $\Gal(M^{\rhobar,\univ}/M^\rhobar) = \calG'$.
We have $\Gal(M^{\rhobar,\univ}/M^\ad) \cong C \times \calG'$ because $C$ is cyclic and the action of $C$ on $\calG'$ by conjugation is trivial
as it corresponds to conjugation by scalar matrices due to Lemma~\ref{lem:split-basis} and Corollary~\ref{cor:univ-inft}.

By Galois theory, there is a unique extension $M^{\ad,\univ}$ of $M^\ad$ contained in $M^{\rhobar,\univ}$ such that $\Gal(M^{\ad,\univ}/M^\ad) \cong \calG'$
as $\FF_p[\Gbar^\ad]$-modules.
The field $M^{\ad,\univ}$ is contained in $M^\ad(S)$ because if a prime of $M^\ad$ ramifies in $M^{\ad,\univ}$, then there is a prime of $M^\rhobar$ above it
that ramifies in $M^{\rhobar,\univ}$ as the orders of $C$ and $\calG'$ are coprime.
Note that $\calG'$ is a quotient of $\calG$ as $\FF_p[\Gbar^\ad]$-modules.
Since, by Proposition~\ref{dimprop}, $I(\chibar/\chibar^\sigma)$ does not occur in $\calG$ as $\FF_p[\Gbar^\ad]$-module, it does not occur in $\calG'$ either.
So, by Corollary~\ref{cor:fratt-ad}, $\Hbar^\ad$ acts trivially on~$\calG'$. This allows us to conclude that $\rho = \rho^\univ_S$ is dihedral
by Theorem~\ref{thm:inf-dih}. 
\end{proof}

\begin{rem}\label{rem:nt-dih}
If $\Hom_{\FF_p[\Gbar^\ad]}(\calG,  I(\chibar/\chibar^{\sigma})) \neq 0$, then we can find an abelian extension $M_0$ of $M^{\rhobar}$ unramified outside primes of $M^{\rhobar}$ lying above $S$ such that $M_0$ is Galois over~$K$ and such that the exact sequence $0 \to \Gal(M_0/M^{\rhobar}) \to \Gal(M_0/K) \to \Image(\rhobar) \to 0$ gives $\Gal(M_0/M^{\rhobar})$ the structure of $\FF_p[\Gbar^\ad]$-module isomorphic to $I(\chibar/\chibar^{\sigma})$.
Therefore, from Proposition~\ref{prop:infi-lift}, we get a non-dihedral infinitesimal deformation of $\rhobar$ relatively unramified outside~$S$.
So, in that case, $\rho^{\univ}_S$ is not dihedral. Thus, from Remark~\ref{classgrprem}, we see that $\rho^{\univ}_S$ is not dihedral if $\Hom_{\FF_p[\Gbar^\ad]}(A(M^{\ad})/pA(M^{\ad}),I(\chibar/\chibar^{\sigma}))$ is non-zero.
It follows, from Remark~\ref{ramrem}, that $\rho^{\univ}_S$ is not dihedral if the set $S$ contains sufficiently many primes from the set $S_0$. Finally, Remark~\ref{prem} implies that if $S$ contains all the primes above $p$ and $\rhobar$ is not totally even, then $\rho^{\univ}_S$ is not dihedral.
\end{rem}

We will now see some consequences of Theorem~\ref{thm:nt-dih}, the hypotheses of which we assume to hold in the sequel.

Let $G^{\ab,(p)}_{M^{\rhobar}(S),L}$ be the maximal continuous pro-$p$ abelian quotient of $\Gal(M^{\rhobar}(S)/L)$. As $S$ does not contain any prime of $K$ above $p$, it follows, from global class field theory, that $G^{\ab,(p)}_{M^{\rhobar}(S),L}$ is a finite, abelian $p$-group. Let $L_S$ be the extension of $L$ such that $\Gal(L_S/L)=G^{\ab,(p)}_{M^{\rhobar}(S),L}$. So, $M^{\rhobar}L_S \subset M^{\rhobar}(S)$. Note that $L_S$ contains the maximal abelian $p$-extension of $L$ unramified outside primes of $L$ lying above~$S$.
Moreover, note that $M^{\rhobar}L_S$ is an abelian extension of $L$ as both $M^{\rhobar}$ and $L_S$ are abelian extensions of $L$. As $p \nmid |\Gal(M^{\rhobar}/L)|$ and $\Gal(L_S/L)$ is a finite abelian $p$-group, it follows that $\Gal(M^{\rhobar}L_S/L) \simeq \Gal(M^{\rhobar}/L) \times \Gal(L_S/L)$. Let $\mathfrak{q}$ be a prime of $L$. If $\mathfrak{q}$ ramifies in $L_S$, then any prime of $M^{\rhobar}$ lying above $\mathfrak{q}$ ramifies in $M^{\rhobar}.L_S$. As $M^{\rhobar}L_S \subset M^{\rhobar}(S)$, it follows that $L_S$ is unramified outside primes of $L$ lying above~$S$. So, it follows that $L_S$ is the maximal abelian $p$-extension of $L$ unramified outside primes of $L$ lying above~$S$.

Suppose $\Gal(L_S/L) = \prod_{i=1}^{n'}\ZZ/p^{e_i}\ZZ$.

\begin{cor}\label{corstr}
Let $S$ be a finite set of primes of $K$. Suppose the conditions given in Theorem~\ref{thm:nt-dih} hold. Then $R^{\univ}_S \simeq W(\FF)[X_1,\cdots,X_{n'}]/((1+X_1)^{p^{e_1}}-1,\cdots,(1+X_{n'})^{p^{e_{n'}}}-1)$.
\end{cor}

\begin{proof}
This follows directly from Theorem~\ref{thm:nt-dih}, part (b) of Theorem~\ref{thm:inf-dih-univ} and Proposition~\ref{prop:univ-char}. 
\end{proof}

Now suppose $p$ is odd. In this case, we add an explicit description of $(R^{\univ}_S)^0$.
Since $L$ is Galois over $K$, $\Gal(M^{\rhobar}(S)/L)$ is a normal subgroup of $\Gal(M^{\rhobar}(S)/K)$ and hence, $L_S$ is Galois over $K$. Now, we get an action of $\Gal(L/K)$ on $\Gal(L_S/L)$ by conjugation. As $p$ is odd, we get a direct sum decomposition $\Gal(L_S/L) = \prod_{i=1}^{n}\ZZ/p^{e_i}\ZZ \oplus  \prod_{i=n+1}^{n'}\ZZ/p^{e_i}\ZZ$ such that $\Gal(L/K)$ acts by inversion on $\prod_{i=1}^{n}\ZZ/p^{e_i}\ZZ$ and trivially on $\prod_{i=n+1}^{n'}\ZZ/p^{e_i}\ZZ$ (note that $n$ could be $0$ or $n'$). 
 
\begin{cor}\label{maincor}
Let $p$ be an odd prime and $S$ be a finite set of primes of $K$. Suppose the conditions given in Theorem~\ref{thm:nt-dih} hold.
Then $(R^{\univ}_S)^0 \simeq W(\FF)[X_1,\cdots,X_n]/((1+X_1)^{p^{e_1}}-1,\cdots,(1+X_n)^{p^{e_n}}-1)$.
\end{cor}

\begin{proof}
It follows, from Theorem~\ref{thm:nt-dih} and part (b) of Theorem~\ref{thm:inf-dih-univ} that $\rho^{\univ}_S = \Ind_{G_{M^{\rhobar}(S)/L}}^{G_{M^{\rhobar}(S)/K}} \chi^{\univ}$ where $\chi^{\univ} : \Gal(M^{\rhobar}(S)/L) \to (R^{\univ}_S)^\times$ is the universal deformation of $\chibar$. As $p$ is odd, it follows, from part (c) of Lemma~\ref{lem:split-basis}, that the exact sequence $0 \to \Gamma^{\univ}_S \to \Image(\rho^{\univ}_S) \to \Image(\rhobar) \to 0$ splits. Note that $\chi^{\univ}$ factors through $\Gal(M^{\rhobar}L_S/L)$. So, from the description of $\chi^{\univ}$ obtained in Proposition~\ref{prop:univ-char}, we get that $\Image(\det(\rho^{\univ}_S))$ is the subgroup of $(R^{\univ}_S)^\times$ generated by $\widehat{\Image(\det{\rhobar})}$ and $(1+X_i)^2$ with $n+1 \leq i \leq n'$. 

We know that $(R^{\univ}_S)^0$ is the quotient $R^{\univ}_S/I$, where $I$ is the ideal generated by the elements $\det(\rho^{\univ}_S(g)) -\widehat{\det(\rhobar(g))}$ for all $g \in \Gal(M^{\rhobar}(S)/K)$ (see \cite[Section $24$]{Ma}). So, $I$ is generated by the elements $X_i(X_i+2)$ with $n+1 \leq i \leq n'$. As $p>2$ and $X_i \in \fm_{R^{\univ}_S}$ for all $1 \leq i \leq n'$, it follows that $I=(X_{n+1},\cdots,X_{n'})$. Therefore, $(R^{\univ}_S)^0 \simeq R^{\univ}_S/(X_{n+1},\cdots,X_{n'}) \simeq W(\FF)[X_1,\cdots,X_n]/((1+X_1)^{p^{e_1}}-1,\cdots,(1+X_n)^{p^{e_n}}-1)$. 
\end{proof}

\subsection{Application to the Boston--Fontaine--Mazur Conjecture}

\begin{proof}[Proof of Corollary~\ref{cor:FM}.]
If the conditions given in Theorem~\ref{thm:nt-dih} hold, then Theorem~\ref{thm:nt-dih} and part (b) of Theorem~\ref{thm:inf-dih-univ} together imply that $\rho^{\univ}_S = \Ind_{G_{M^{\rhobar}(S)/L}}^{G_{M^{\rhobar}(S)/K}} (\chi^{\univ})$ where $\chi^{\univ} : \Gal(M^{\rhobar}(S)/L) \to (R^{\univ}_S)^\times$ is the universal deformation of $\chibar$. It follows, from the description of $\chi^{\univ}$ in Proposition~\ref{prop:univ-char} and the discussion before Corollary~\ref{corstr}, that $\chi^{\univ}$ has finite image. Therefore, from the discussion before Definition~\ref{defi:modules}, we see that $\rho^{\univ}_S$ has finite image. 
\end{proof}

\begin{proof}[Proof of Corollary~\ref{cor:exm}]
In order to prove this corollary, we will check that the set $S$ satisfies the hypotheses of Theorem~\ref{thm:nt-dih} in both the cases.
Recall that we denoted by $M$ the maximal unramified abelian $5$-extension of~$L$, which exists as the class number of $L$ is~$15$.
Now, the class number of $M$ is $3$ (see \cite[\href{http://www.lmfdb.org/NumberField/20.0.35908028125401873392383429449.1}{Number field 20.0.35908028125401873392383429449.1}]{lmfdb}). 

Let us assume that we are in case~(a) of the corollary. So, we have $K=\QQ(\sqrt{717})$, $L=\QQ(\sqrt{-3},\sqrt{-239})$, $M^{\ad}=M$ and $\Gbar^{\ad} = \Gal(M/\QQ(\sqrt{717})) \simeq D_5$.
As the class number of $M$ is $3$, it follows that $\Hom_{\FF_3[\Gbar^{\ad}]}(A(M)/3A(M), I(\chi/\chi^{\sigma}))=0$.
Now, $S_{\infty} \subseteq S$ and $S \cap S_p = \emptyset$. Let $P$ be the set of all finite primes of $\QQ(\sqrt{717})$.
Now, as $\mu_3 \subset \QQ(\sqrt{-3},\sqrt{-239})$, we see that $S_1=\{\ell \in P \;|\; \ell \text{ is totally split in }M\}$ and $S_2 = \{\ell \in P \;|\; \ell \text{ is inert in }L\}$.
Since we are working with $D_5$, $S_3 = \emptyset$. Now, if $\ell$ is a finite prime contained in $S$, then $\ell$ is split in $L$, which means that $\ell \not\in S_2$.
But $\ell$ is not completely split in $M$, which means that $\ell \not\in S_1$.
So, $S \cap S_0 = \emptyset$.
Hence, all the hypotheses of Theorem~\ref{thm:nt-dih} are satisfied.
Therefore, by Corollary~\ref{cor:FM}, we get that the universal deformation of $\rhobar_1$ relatively unramified outside $S$ has finite image.

The proof in the other case follows in the exact same way. In that case $\mu_3 \subset \QQ(\sqrt{-3}) =K$ and we have already given the description of $S_0$ in the case $\mu_p \subset K$. So, we can use that description to prove that $S \cap S_0 =\emptyset$. 
\end{proof}

\begin{rem}
In the introduction, we said that the examples of Corollary~\ref{cor:exm} do not satisfy the hypotheses of \cite[Corollary 3]{AC}. Here, we would like to elaborate a bit more on that.
Allen and Calegari prove, in \cite[Corollary 3]{AC}, that if $F$ is a totally real field and $\rho : G_F \to \GL_2(\FF)$ is a totally odd representation satisfying certain hypotheses, then Boston's conjecture is true for $\rho$.
As the base field is assumed to be totally real in \cite[Corollary 3]{AC}, Corollary~\ref{cor:exm}\eqref{cor:exm:b} clearly does not satisfy those hypotheses.
Moreover, one of the hypotheses of \cite[Corollary 3]{AC} is that the image of $\rho|_{G_{F(\mu_p)}}$ is adequate.
However, we see that the image of $\rhobar_1|_{G_{\QQ(\sqrt{717},\mu_3)}}$ is just $\rhobar_1(G_{\QQ(\sqrt{-3},\sqrt{-239})})$ which is an abelian group. Hence, it is not adequate, which means that Corollary~\ref{cor:exm}\eqref{cor:exm:a} does not satisfy the hypotheses of \cite[Corollary 3]{AC} either.
\end{rem}

\subsection{Application to a question of Greenberg and Coleman}

The main purpose of this section is to prove a result similar to \cite[Corollary $1.3.2$]{CWE} using the main results of this article.

\begin{setup}\label{setup:app}
Let $p >2$ be a prime and $\FF$ be a finite field of characteristic~$p$.
We keep Notation~\ref{not:nt} and Set-up~\ref{setup:dih} throughout this section with $K=\QQ$. Let $\rhobar : G_{\QQ} \to \GL_2(\FF)$ be an absolutely irreducible, odd dihedral representation. So, $\rhobar = \Ind_{G_L}^{G_{\QQ}} (\chibar)$ for some quadratic extension $L$ of $\QQ$ and character $\chibar : G_L \to \FF^{\times}$. Let $N$ be the tame Artin conductor of $\rhobar$. Moreover, assume that
\begin{enumerate}
\item\label{app:im} $L$ is a quadratic imaginary number field,
\item\label{app:split} $p$ is split in $L$,
\item\label{app:distinct} $\rhobar|_{G_{\QQ_p}}$ is a sum of two \emph{distinct} characters,
\item\label{app:ram} If $\ell \neq p$ is a prime such that $\rhobar$ is ramified at $\ell$ and $\rhobar|_{G_{\QQ_{\ell}}} = \eta \oplus \psi$, then $\eta\psi^{-1} \neq \omega^{(\ell)}_p, (\omega^{(\ell)}_p)^{-1}$, where $\omega^{(\ell)}_p$ is the mod $p$ cyclotomic character of $G_{\QQ_{\ell}}$,
\item\label{app:cond} If $\ell$ is a prime such that $p | \ell-1$, $\rhobar$ is ramified at $\ell$, $\ell$ is not split in $L$ and $\rhobar|_{G_{\QQ_{\ell}}}$ is reducible, then $\rhobar|_{G_{\QQ_{\ell}}}$ is a sum of a ramified and an unramified character.
\item\label{app:irred} If $\ell$ is a prime such that $p | \ell+1$ and $\ell$ is ramified in $L$, then $\ad(\rhobar)(G_{\QQ_\ell})  \not\simeq \ZZ/2\ZZ \times \ZZ/2\ZZ$.
\end{enumerate}
\end{setup}

Let $S_0$ be the set of primes of $\QQ$ defined just before Proposition~\ref{dimprop}. Let $S$ be a finite set of primes of $\QQ$ such that $S_{\rhobar} \cup \{p\} \subset S$. Let $T := S \setminus (S \cap S_0)$. Let $G^{p-\ab}_S$ be the maximal continuous quotient of $\Gal(M^{\rhobar}(T)/\QQ)$ in which the image of $G_{\QQ_p}$ is abelian. Note that $G^{p-\ab}_S$ satisfies the finiteness condition $\Phi_p$ of Mazur and $\rhobar$ factors through $G^{p-\ab}_S$ due to assumption~\ref{app:split} above. After considering $\rhobar$ as a representation of $G^{p-\ab}_S$, let $R^{\univ,S}$ be the universal deformation ring of~$\rhobar$ with universal deformation~$\rho^{\univ,p-\ab}_S$.

\begin{prop}\label{app:dih}
Let $S$ be a finite set of primes of $\QQ$ such that $S_{\rhobar} \cup \{p\} \subset S$.\\ 
If ~$\Hom_{\FF_p[\Gbar^\ad]}(A(M^{\ad})/pA(M^{\ad}) ,  I(\chibar/\chibar^{\sigma})) = 0$, then $\rho^{\univ,p-\ab}_S$ is dihedral.
\end{prop}

\begin{proof}
By Theorem~\ref{thm:inf-dih-univ}, it suffices to prove that any deformation $\rho : G^{p-\ab}_S \to \GL_2(\FF[X]/(X^2))$ of $\rhobar$ is dihedral.
Now the existence of a non-dihedral deformation $\rho : G^{p-\ab}_S \to \GL_2(\FF[X]/(X^2))$ of $\rhobar$ implies that $H^1(G^{p-\ab}_S,I(\chibar/\chibar^{\sigma})) \neq 0$.
So we will prove that all infinitesimal $G^{p-\ab}_S$-deformations of $\rhobar$ are dihedral by proving that $H^1(G^{p-\ab}_S,I(\chibar/\chibar^{\sigma}))=0$. 

Suppose  $H^1(G^{p-\ab}_S,I(\chibar/\chibar^{\sigma})) \neq 0$ and let $\phi$ be a non-zero element of  $H^1(G^{p-\ab}_S,I(\chibar/\chibar^{\sigma}))$.
Let $\rho : G^{p-\ab}_S \to \GL_2(\FF[X]/(X^2))$ be a deformation of $\rhobar$ corresponding to $\phi$ and let $M^{\rho}$ be the extension of $\QQ$ fixed by $\ker(\rho)$.

Note that $M^{\rho}$ is an abelian $p$-extension of $M^{\rhobar}$ unramified outside primes of $M^{\rhobar}$ lying above $T$ and $\Gal(M^{\rho}/M^{\rhobar}) \simeq I(\chibar/\chibar^{\sigma})$ as $\FF_p[\Gbar]$-modules.
As the image of $G_{\QQ_p}$ in $\Gal(M^{\rho}/\QQ)$ is abelian, it follows, from the assumptions \ref{app:split} and \ref{app:distinct}, that any prime of $M^{\rhobar}$ lying above $p$ splits completely in $M^{\rho}$.
Hence, $M^{\rho} \subset M^{\rhobar}(T \setminus \{p\})$. Denote $T \setminus \{p\}$ by $T'$.

From the proof of Theorem~\ref{thm:nt-dih}, it follows that there exists an abelian $p$-extension $M^{\ad,\rho}$ of $M^\ad$ such that $M^{\rho}=M^{\rhobar}M^{\ad,\rho}$, $M^{\ad,\rho} \subset M^{\ad}(T')$ and $\Gal(M^{\ad,\rho}/M^{\ad}) \simeq I(\chibar/\chibar^{\sigma})$ as $\FF_p[\Gbar^{\ad}]$-modules.
This implies that $\Hom_{\FF_p[\Gbar^{\ad}]}(G^{\ab}_{M^{\ad},T'}/(G^{\ab}_{M^{\ad},T'})^p, I(\chibar/\chibar^{\sigma})) \neq 0$.
However, as $p \not\in T'$, $T' \cap S_0 =\emptyset$ and $\Hom_{\FF_p[\Gbar^\ad]}(A(M^{\ad})/pA(M^{\ad}) ,  I(\chibar/\chibar^{\sigma})) = 0$, we get a contradiction from Proposition~\ref{dimprop}. Hence, it follows that $H^1(G^{p-\ab}_S,I(\chibar/\chibar^{\sigma}))=0$ which implies the proposition. 
\end{proof}

\begin{rem}
Note that $\rho^{\univ,p-\ab}_S$ is the universal nearly ordinary $p$-locally split representation (in the sense of \cite{GV11}) relatively unramified outside $S \setminus (S \cap S_0)$.
\end{rem}

If $f$ is a classical modular eigenform of level $\Gamma_1(M)$, denote the $p$-adic Galois representation attached to it by $\rho_f$ and the corresponding semi-simple residual representation by $\rhobar_f$. We now prove Theorem~\ref{thm:CM2}.

\begin{proof}[Proof of Theorem~\ref{thm:CM2}.]
Let $S'$ be the set of prime divisors of $M$. Since $\rho_f$ is a lift of $\rhobar$, by Proposition~\ref{app:dih} with $S=S' \cup \{p\}$, it suffices to prove that $\rho_f$ factors through $G^{p-\ab}_S$.
So we need to prove that there is no ramification above $M^\rhobar$ above primes in $S' \cap S_0$ and that the image of $G_{\QQ_p}$ is abelian.
The latter is satisfied because $\rho_f|_{G_{\QQ_p}}$ is a sum of two characters.
Hence, we only need to prove that if $\ell \in S' \cap S_0$, then $\rho_f(I_{\ell}) \simeq \rhobar(I_{\ell})$.

Suppose $\ell \in S'$ is split in $L$. So we have $\rhobar|_{G_{\QQ_{\ell}}} = \chibar|_{G_{\QQ_\ell}} \oplus \chibar^{\sigma}|_{G_{\QQ_\ell}}$. By the third hypothesis of the theorem, we get that $\chibar/\chibar^{\sigma}|_{G_{\QQ_\ell}} \neq \omega_p^{(\ell)}, (\omega_p^{(\ell)})^{-1}$. Therefore, we get that $\ell \not\in S_1$ and hence, $S' \cap S_1 =\emptyset$.

Note that for $\ell \in S_2$, we have $p | \ell-1$ because the third hypothesis of the theorem excludes $L_\lambda = \QQ_\ell(\mu_p)$.
By Assumption \ref{app:ram} and \ref{app:cond}, along with the fourth hypothesis of the theorem, we get
$$S' \cap S_2 = \{\ell | M \text{ such that } p | \ell-1, \ell \text{ is ramified in }L \text{ and } |\ad(\rhobar)(D_{\ell})|=2\}.$$
Indeed the definition of $S_2$, along with Assumption \ref{app:ram}, implies that if $\ell \in S' \cap S_2$, then $\ell | M$, $p | \ell-1$, $\ell$ is not split in $L$ and $|\ad(\rhobar)(D_{\ell})|=2$. Now the assumption $\ell \nmid M/N$ and \ref{app:cond} implies that $\ell$ is ramified in $L$.

Let now $\ell \in S' \cap S_2$.
Then $\rhobar|_{G_{\QQ_{\ell}}} = \eta \oplus \epsilon\eta$ where $\eta$ is an unramified character and $\epsilon$ is the character corresponding to a ramified quadratic extension of $\QQ_{\ell}$. The unramifiedness of $\eta$ follows from Assumption~\ref{app:cond}.
So, we get that $\ell |N$ but $\ell^2 \nmid N$. Hence, from the fourth hypothesis of the theorem, we get that $\ell^2 \nmid M$. Hence, $\rho_f|_{G_{\QQ_\ell}}$ is not irreducible. 
As $\epsilon \neq 1$, we see, by the local Langlands correspondence, that $\rho_f|_{G_{\QQ_{\ell}}}$ is a sum of two characters $\chi_1$ and $\chi_2$ lifting $\eta$ and $\epsilon\eta$, respectively.
As $M$ is the tame Artin conductor of $\rho_f$, it follows that $\chi_1$ is unramified at~$\ell$.
As $\det(\rho_f|_{I_{\ell}}) = \widehat{\det(\rhobar|_{I_{\ell}})}$, it follows that $\rho_f|_{I_{\ell}} \simeq \rhobar|_{I_{\ell}}$.

Finally, suppose $S' \cap S_3 \neq \emptyset$ and let $\ell \in S' \cap S_3$.
Then $p | \ell+1$, $\ell$ is ramified in $L$ and $\ad(\rhobar)(G_{\QQ_\ell}) \simeq \ZZ/2\ZZ \times \ZZ/2\ZZ$.
But this contradicts the hypothesis \eqref{app:irred}. Hence, we get that $S' \cap S_3 = \emptyset$. 
\end{proof}

\begin{rem}\label{rem:CM}
Note that the conditions given in Set-up~\eqref{setup:app} on $\rhobar$ and the conditions in Theorem~\ref{thm:CM2} are slightly weaker than the conditions appearing in \cite[Corollary $1.3.2$]{CWE}. Indeed, in \cite[Corollary $1.3.2$]{CWE}, they assume that the tame level of the modular form is the same as the tame Artin conductor of $\rhobar$. The assumption $\Hom_{\FF_p[\Gbar^\ad]}(A(M^{\ad})/pA(M^{\ad}) ,  I(\chibar/\chibar^{\sigma})) = 0$ appearing in Theorem~\ref{thm:CM2} is the same as the assumption $X(\psi^-)=0$ appearing in \cite[Corollary $1.3.2$]{CWE}. The asumptions \eqref{app:im}, \eqref{app:split} and \eqref{app:distinct} of Set-up~\eqref{setup:app} are present in \cite[Corollary $1.3.2$]{CWE}.
Observe that the assumptions \eqref{app:ram}, \eqref{app:cond} and \eqref{app:irred} of Set-up~\eqref{setup:app} are satisfied when $\chibar$ is a character of $G_L$ such that the conductors of $\chibar$ and $\chibar^{\sigma}$ are coprime. Indeed, if the conductors of $\chibar$ and $\chibar^{\sigma}$ are co-prime, then we get that if $\ell | N$, then $\rhobar|_{G_{\QQ_\ell}} \simeq \phi \oplus \eta$, where exactly one of $\phi$ and $\eta$ is a ramified character (see \cite[Lemma $2.3.4$]{CWE}). This assumption on $\chibar$ is also present in \cite{CWE} (see \cite[Section $1.2.3$]{CWE}). Hence, Theorem~\ref{thm:CM2}  implies \cite[Corollary $1.3.2$]{CWE}.
\end{rem}

\begin{rem}
The arguments of \cite{CWE} rely crucially on an '$R=\TT$' theorem that they prove (\cite[Theorem $5.5.1$]{CWE}). However, our approach relies mainly on the group theoretic arguments and on proving that universal any nearly ordinary $p$-locally split deformation of $\rhobar$ is dihedral. However, if the universal nearly ordinary $p$-locally split deformation of $\rhobar$ is not dihedral, we cannot say anything conclusive regarding the question of Greenberg. So our argument does not yield anything when $\Hom_{\FF_p[\Gbar^\ad]}(A(M^{\ad})/pA(M^{\ad}) ,  I(\chibar/\chibar^{\sigma})) \neq 0$. Hence, we cannot prove \cite[Theorem $1.3.4$, Theorem $1.4.4$]{CWE} using our methods.
\end{rem}

\begin{rem}
Suppose that in Set-up~\ref{setup:app}, we assume that $L$ is a real quadratic field, but that all the other assumptions hold. Then the above arguments shows that if $f$ is an eigenform satisfying the hypotheses of Theorem~\ref{thm:CM2}, then $\rho_f$ is dihedral and induced from a character of~$G_L$. But as $L$ is a real quadratic field, this would immediately imply that $f$ has weight $1$ as there is no Real Multiplication in weights greater than~$1$.

This proves that if $f$ has weight greater than 1 and satisfies the hypotheses 1, 3 and 4 of Theorem~\ref{thm:CM2}, then $\rho_f|_{G_{\QQ_p}}$ is not a sum of two characters and $f$ does not have complex multiplication. This answers the question of Coleman-Greenberg in some more cases.
\end{rem}

\section{Modularity and an $R=\TT$-theorem}

In this section, we let $K$ be a totally real field and $L/K$ a quadratic extension, which is not necessarily totally imaginary.
We furthermore fix an absolutely irreducible dihedral representation $\rhobar: G_K \to \GL_2(\FF)$ which we assume to be totally odd.
As before, let $R$  be a complete Noetherian local ring with residue field $\FF$ and let $\rho: G_K \to \GL_2(R)$ be a dihedral deformation of~$\rhobar$ of the form $\Ind_{G_L}^{G_K}(\chi)$.
We assume $\rho$ to be unramified at the primes above~$p$ and of finite image. We write $H = G_L$ and $G=G_K$.

The first aim of this section is to prove that $\rho$ is modular of parallel weight one. Below a comparison theorem with a universal deformation ring
is deduced under certain conditions.
The modularity is based on the following well-known lemma.

\begin{lem}\label{modlem}
Let $\chi : G_L \to (\overline{\QQ_p})^\times$ be a finite order character of $G_L$ such that it is unramified at places dividing~$p$ and such that the induced representation $\rho=\Ind_{G_L}^{G_K}(\chi)$ is a totally odd, absolutely irreducible representation of $G_K$. Let $D$ be the tame Artin conductor of~$\rho$. Then there exists a Hilbert modular eigenform $f$ over $K$ of parallel weight one of level $\Gamma_1(D)$ such that the Galois representation $\rho_f$ attached to $f$ is isomorphic to~$\rho$.
\end{lem}

\begin{proof}
The existence of the Hilbert modular eigenform of parallel weight one over $K$ follows from the proofs of \cite[Lemma $4.9$ and Lemma $4.10$]{O}, which uses automorphic induction, and \cite[Theorem $1.4$]{RaTa} (see \cite[Section $5.3(A)$ and Theorem $5.3.1$]{G} or \cite[Theorem $17$]{R} as well). The assertion for the level of $f$ follows from \cite[Theorem $1.4$]{RaTa} and the local-global compatibility in the Langlands correspondence. When $L/K$ is a CM field, the entire lemma also follows from \cite{H} (see \cite[Section $1$]{BGV}). 
\end{proof}

Let $S_1(\Gamma_1(\fn),\Qbar_p)$ denote the space of parallel weight one cuspidal Hilbert modular forms over~$K$ of level $\Gamma_1(\fn)$ with coefficients in $\Qbar_p$ (for
example, via a fixed isomorphism $\CC \cong \Qbar_p$).
Further, let $\TT_{1,\fn}$ be the $W(\FF)$-algebra generated by the Hecke operators $T_\fq$ for primes $\fq \nmid \fn p$ acting faithfully on $S_1(\Gamma_1(\fn),\Qbar_p)$.

\begin{prop}\label{prop:Rmodular}
Let $\rhobar: G_K \to \GL_2(\FF)$ be an absolutely irreducible, totally odd, dihedral representation. Let $R$  be a complete Noetherian local ring with residue field $\FF$ and let $\rho: G_K \to \GL_2(R)$ be a dihedral deformation of~$\rhobar$ with finite image.
Suppose $\rho$ is unramified at all the primes above $p$.
Then there is an ideal $\fn_\rho$ of $K$ coprime to~$p$ and a quotient $\TT_{\dih,\rho}$ of~$\TT_{1,\fn_\rho}$, as well as a Galois representation $\rho_{\dih}: G_K \to \GL_2(\TT_{\dih,\rho})$ and a
homomorphism $\TT_{\dih,\rho} \to R$ such that
\begin{enumerate}
\item $\TT_{\dih,\rho}$ is a complete Noetherian local $W(\FF)$-algebra with residue field $\FF$,
\item $\rho_{\dih}$ is a deformation of $\rhobar$,
\item $\rho$ factors as
$$ \rho: G_K \xrightarrow{\rho_{\dih}} \GL_2(\TT_{\dih,\rho}) \to \GL_2(R).$$
\end{enumerate}
The Galois representation $\rho_{\dih}$ is unramified outside $p \fn_\rho$ and characterised by the property
$$\Tr(\rho_{\dih}(\Frob_\fq))=T_\fq$$
for all primes $\fq \nmid p \fn_\rho$.
\end{prop}

\begin{proof} 
We will refer to $\TT_{\dih,\rho}$ as just $\TT_\dih$ throughout the proof for the ease of notation.
The strategy is to factor the representation $\rho$ via a ring $S$ which will contain the Hecke algebra $\TT_\dih$. We fix a field isomorphism $\Qbar_p \cong \CC$.
Since $\rho$ is assumed to be dihedral, there exists a quadratic extension $L$ of $K$ and a character $\chi : G_L \to R^{\times}$ such that $\rho = \Ind_{G_L}^{G_K}(\chi)$.
We first factor $\chi$ uniquely as a product $\chi = \chibarhat \cdot \psi$, where $\chibarhat$ has prime-to~$p$ order
(equal to the order of $\chibar$ and dividing the order of $\FF^\times$)
and the order of $\psi$ is a power of~$p$. As $\rho$ is assumed to have finite image, we see that $\psi$ is a character of finite order.

Let $U \subseteq R^\times$ be the image of~$\psi$, that is, $U$ is a finite $p$-group. Let $S := W(\FF)[U]$ be the group ring over the Witt vectors of~$\FF$.
By the universal property of Witt vectors, we have a natural morphism $\epsilon: W(\FF) \to R$. Consequently, $\chibarhat$ factors via~$\epsilon$:
$$ \chibarhat: H \xrightarrow{\chibarhat} W(\FF) \xrightarrow{\epsilon} R^\times.$$
Furthermore, by the universal property of group rings, we obtain a unique morphism $\delta: W(\FF)[U] \to R$ which is the identity on~$U$ and given by $\epsilon$ on $W(\FF)$. We can hence also factor $\psi$ as
$$ \psi: H \to U \xrightarrow{\incl} W(\FF)[U]^\times \xrightarrow{\delta} R^\times.$$
Denote by $\psi_S: H \to W(\FF)[U]^\times$ the composition of the first two maps.
Defining $\chi_S := \chibarhat \cdot \psi_S : H \to W(\FF)[U]$, allows us to factor $\chi$ as
$$ \chi : H \xrightarrow{\chi_S} W(\FF)[U]^\times \xrightarrow{\delta} R^\times.$$
We will now factor this map at $W(\FF)[U]$ towards a Hecke algebra.

In order to do so, we consider the homomorphism of $W(\FF)[U]$-algebras
$$ \lambda: W(\FF)[U] \hookrightarrow \Qbar_p[U] \xrightarrow{ \sum_{u \in U} r_u u \mapsto (\sum_{u \in U} r_u \alpha(u))_\alpha} \prod_{\alpha} \Qbar_p^{(\alpha)},$$
where $\alpha$ runs through all $\Qbar_p^\times$-valued characters of the abelian group~$U$ and $\Qbar_p^{(\alpha)}$ is the unique $\Qbar_p$-vector space of dimension~$1$
with $U$-action via the character~$\alpha$. The first map is the natural inclusion and the second one is an isomorphism by standard characteristic zero representation theory of finite groups.

We first factor at $W(\FF)[U]$ using any of the characters $\alpha$ individually by letting
$$\rho_\alpha := \Ind_H^G(\chibarhat \cdot (\alpha \circ \psi_S)): G \to \GL_2(\Qbar_p).$$
This is a deformation of~$\rhobar$.
Let $\fn_\alpha$ be the tame Artin conductor of $\rho_\alpha$. Then $\fn_\alpha$ is coprime to~$p$ and, as $\rho$ is unramified at all primes above $p$, $\rho_{\alpha}$ is unramified outside primes dividing $n_{\alpha}\infty$. Moreover, we have for all prime ideals $\fq \nmid \fn_\alpha$:
$$ \Tr(\rho_\alpha(\Frob_\fq)) = \begin{cases}
\chibarhat(\Frob_{\fq_1}) \cdot \alpha(\psi_S(\Frob_{\fq_1})) + \chibarhat(\Frob_{\fq_2}) \cdot \alpha(\psi_S(\Frob_{\fq_2})) & \textnormal{ if } \fq \calO_L = \fq_1 \fq_2,\\
0 & \textnormal{ otherwise.} \end{cases}$$
By Lemma~\ref{modlem}, there is a unique cuspidal Hilbert modular eigenform $f_\alpha \in S_1(\Gamma_1(\fn_\alpha),\Qbar_p)$, the cuspidal parallel weight one space
of level $\Gamma_1(\fn_\alpha)$, such that $T_\fq f_\alpha =  \Tr(\rho_\alpha(\Frob_\fq)) \cdot f_\alpha$ for all prime ideals $\fq \nmid \fn_\alpha$.

We now use standard arguments to `interpolate all these $f_\alpha$'s' and introduce Hecke algebras for that reason.
Let $\fn$ be the least common multiple of the ideals $\fn_\alpha$, running over all characters~$\alpha$.
Let $W = \langle f_\alpha \;|\; \alpha: U \to \Qbar_p^\times \rangle_{\Qbar_p}$ be the $\Qbar_p$-vector subspace of $S_1(\Gamma_1(\fn),\Qbar_p)$ generated by the~$f_\alpha$.
It is stable under all Hecke operators $T_\fq$ for $\fq$ running through the primes not dividing $p\fn$. Let $\TT_\dih$ be the quotient of $\TT_{1,\fn}$ acting faithfully on~$W$.

Let $\calO$ be the ring of integers of a finite extension of $\QQ_p$ containing the $T_\fq$-eigenvalues of all $f_{\alpha} \in W$ for all primes $\fq \nmid \fn p$ and let $\fm$ be its maximal ideal. Now $\TT_{\dih}$ is a subring of the $\calO$-linear endomorphisms of the $\calO$-span of all $f_\alpha \in W$. Since $\calO$ is finite over $W(\FF)$ and the $\calO$-span of all $f_\alpha \in W$ is finite over $\calO$, we see that $\TT_\dih$ is finite over $W(\FF)$. Hence, it is a complete Noetherian $W(\FF)$-algebra.  

Since for any prime $\fq \nmid \fn p$, the $T_\fq$-eigenvalues of all $f_{\alpha}$'s are congruent modulo $\fm$, it follows, from the Deligne--Serre lemma, that $\TT_\dih$ is local. Let $\fm'$ be its maximal ideal.  Now, for any prime $\fq \nmid \fn p$, the image of $T_\fq$ modulo $\fm'$ is the same as the image of the $T_\fq$-eigenvalue of $f_{\alpha}$ modulo $\fm$ for any $\alpha$. Hence, the image of $T_\fq$ modulo $\fm'$ is $\Tr(\rhobar(\Frob_\fq))$ for all primes $\fq \nmid \fn p$ and, hence, lies in $\FF$. Therefore, it follows that the residue field of $\TT_\dih$ is $\FF$.

In order to describe $\TT_\dih$ explicitly, we use that `multiplicity one' applied to the newforms $f_\alpha$ provides an injective $W(\FF)$-algebra homomorphism
$$ \beta : \TT_\dih \xrightarrow{T_\fq \mapsto \big(\Tr(\rho_\alpha(\Frob_\fq))\big)_\alpha} \prod_{\alpha} \Qbar_p^{(\alpha)}.$$
By the description of $\Tr(\rho_\alpha(\Frob_\fq))$ above we obtain the crucial containment
$$ \Image(\beta) \subseteq \Image(\lambda).$$
We may thus factor $\beta$ as
$$ \beta: \TT_\dih \xrightarrow{\beta'} W(\FF)[U] \xrightarrow{\lambda} \prod_{\alpha} \Qbar_p^{(\alpha)}.$$
In view of the above formulas, we explicitly have
$$ \beta'(T_\fq)=\chibarhat(\Frob_{\fq_1}) \cdot \psi_S(\Frob_{\fq_1}) + \chibarhat(\Frob_{\fq_2}) \cdot \psi_S(\Frob_{\fq_2})
= \chi_S(\Frob_{\fq_1})+\chi_S(\Frob_{\fq_2})$$
if $\fq \calO_L = \fq_1 \fq_2$ and $\beta'(T_\fq)=0$ otherwise.
Consequently, $\delta(\beta'(T_\fq))=\Tr(\rho(\Frob_\fq))$ for all primes $\fq$ not dividing~$\fn p$.

Now, the representation $\rho_{\alpha}$ takes values in $\GL_2(\calO)$ for all $f_{\alpha} \in W$ and $\Image(\beta)$ also lies in $\prod_{\alpha}\calO^{(\alpha)}$. Consider $\rho'=\prod_{\alpha}\rho_{\alpha} : G_K \to \prod_{\alpha} \GL_2(\calO)$.
Note that $\rho_{\alpha} \pmod{\fm} = \rhobar$ is absolutely irreducible. Now by the Chebotarev density theorem, $\Tr(\rho'(g)) \in \Image(\beta)$ for all $g \in G_K$.
Hence, applying \cite[Th\'eor\`eme~2]{carayol} to $\rho'$, we obtain a Galois representation
$$ \rho_{\dih}: G_K \to \GL_2(\TT_\dih)$$
which is unramified outside $\fn$ and satisfies $\Tr(\rho_{\dih}(\Frob_\fq))=T_\fq$ for all prime ideals $\fq$ not dividing~$p\fn$.

Finally, we consider the composition
$$r : G_K \xrightarrow{\rho_{\dih}} \GL_2(\TT_\dih)\xrightarrow{\beta'} \GL_2(W(\FF)[U]) \xrightarrow{\delta} \GL_2(R).$$
We see $\Tr(r(\Frob_\fq)) = \Tr(\rho(\Frob_\fq))$ as well as $\Tr(\rho_{\dih}(\Frob_\fq)) \pmod{\fm'}= \Tr(\rhobar(\Frob_\fq))$ for all primes $\fq \nmid p \fn$.
By \cite[Th\'eor\`eme~1]{carayol} we obtain that $r$ and $\rho$ are isomorphic and $\rho_{\dih} \pmod{\fm'}$ and $\rhobar$ are isomorphic, proving the proposition. 
\end{proof}

We next apply this result to prove an `$R=\TT$' theorem for the case that the universal deformation is dihedral. We first strive at a large generality.

\begin{defi}\label{def:P}
Let {\bf P} be a property that a deformation of $\rhobar$ may (or not) satisfy. We say that {\bf P} is a deformation condition if $\rhobar$ satisfies {\bf P}, {\bf P} only depends
on the equivalence class of the representation and there exists a deformation $\rho^\univ_{\bf P}: G \to \GL_2(R^\univ_{\bf P})$ of~$\rhobar$ that is universal for the property {\bf P} in the sense that for exactly those
$\rho:G \to \GL_2(R)$ having property {\bf P}, there is a unique morphism $\phi: R^\univ_{\bf P} \to R$ such that $\rho = \phi \circ \rho^\univ_{\bf P}$.
\end{defi}

See \cite[Section 23]{Ma} for an alternate description of deformation conditions.
\begin{thm}\label{thm:ReqT}
Let $\bf P$ be a deformation condition in the sense of Definition~\ref{def:P}.
Suppose that $\rho^\univ_{\bf P}$ is dihedral and unramified above~$p$.
Let $\TT_\dih := \TT_{\dih,\rho^\univ_{\bf P}}$ be the Hecke algebra of parallel weight one attached to $\rho^\univ_{\bf P}$ by Proposition~\ref{prop:Rmodular}. Let $\psi: \TT_\dih\to R^\univ_{\bf P}$
and $\rho_\dih$ be the associated ring homomorphism and Galois representation, respectively. Suppose further that $\rho_\dih$ also satisfies {\bf P}.

Then $R^\univ_{\bf P} \cong \TT_\dih$.
\end{thm}

\begin{proof}
As $\rho^\univ_{\bf P}$ is dihedral and unramified at $p$, it follows, by global class field theory, that it has finite image (see Proposition~\ref{prop:univ-char} and the discussion before Corollary~\ref{corstr}). Hence, we can appeal to  Proposition~\ref{prop:Rmodular}. Since $\rho_\dih$ satisfies {\bf P}, there is a morphism $\phi: R^\univ_{\bf P} \to \TT_\dih$ such that $\rho_\dih = \phi \circ \rho^\univ_{\bf P}$.
As $\TT_\dih$ is generated by $T_\fq = \Tr(\rho_{\TT_\dih}(\Frob_\fq)) = \phi(\Tr(\rho^\univ_{\bf P}(\Frob_\fq)))$ for primes $\fq\nmid p\fn$, the map $\phi$ is surjective.
By Proposition~\ref{prop:Rmodular}, we know that $\psi \circ \rho_\dih = \rho^\univ_{\bf P}$.
So, the composition $R^\univ_{\bf P} \xrightarrow{\phi} \TT_\dih \xrightarrow{\psi} R^\univ_{\bf P}$ sends $\rho^{\univ}_{\bf P}$ to itself, {\it i.e.} $\psi \circ \phi : R^\univ_{\bf P} \to R^\univ_{\bf P}$ is a morphism such that $\psi \circ \phi \circ \rho^\univ_{\bf P} = \rho^\univ_{\bf P}$.
Hence, the universality of $R^\univ_{\bf P}$ implies that $\psi \circ \phi$ is the identity morphism.
This proves that $\phi$ is also injective. 
\end{proof}

\begin{setup}\label{setup:modularity}
We continue to assume Set-up~\ref{setup:dih} in the above setting.
Thus, we have a number field $K$, a quadratic extension $L$ of $K$, a finite extension $\FF$ of $\FF_p$ and a character $\chibar : G_L \to \FF^{\times}$ such that the representation $\rhobar=\Ind_{G_L}^{G_K}(\chibar) : G_K \to \GL_2(\FF)$ is absolutely irreducible.
Let $D$ be the tame Artin conductor of~$\rhobar$.

For the rest of this section, we specialise Set-up~\ref{setup:dih} as indicated in section~\ref{sec:intro-modularity}:
\begin{enumerate}
\item $p$ is odd.
\item \label{ass:1} $K$ is totally real.
\item \label{ass:2} $\chibar$ is such that $\rhobar$ is totally odd.
\item \label{ass:4}$\rhobar$ is unramified at all places of $K$ above $p$, i.e.\ $S_{\rhobar} \cap S_p = \emptyset$.
\item\label{ass:5} If a prime $\ell$ of $K$ ramifies in $M^{\rhobar}$ and $\rhobar|_{G_{K_{\ell}}}$ is not absolutely irreducible, then $\dim((\rhobar)^{I_{\ell}}) =1$ where $(\rhobar)^{I_{\ell}}$ denotes the subspace of $\rhobar$ fixed by the inertia group $I_{\ell}$ at $\ell$.
\item \label{ass:3} $\Hom_{\FF_p[\Gbar^\ad]}(A(M^{\ad})/pA(M^{\ad}) ,I(\chibar/\chibar^{\sigma})) =0$.
\end{enumerate}
For $K = \QQ$, conditions \ref{ass:4} and \ref{ass:5} are the ones given in \cite[Section $3.1$]{CG}.
\end{setup}

A prime~$\ell$ of $K$ is called a {\em vexing prime} if $\rhobar$ ramifies at~$\ell$, $\rhobar|_{G_{K_{\ell}}}$ is absolutely irreducible, $\rhobar|_{I_{\ell}}$ is \emph{not} absolutely irreducible and $[K_{\ell}(\mu_p) : K_{\ell}]=2$.
Note that the definition of vexing primes makes sense for any absolutely irreducible $\rhobar : G_K \to \GL_2(\FF)$.
Recall that for an absolutely irreducible $\rhobar :G_K \to \GL_2(\FF)$ and a finite set of primes $S$ of $K$, we have defined $R_S^\univ$ (resp. $\rho_S^{\univ}$) to be the universal deformation ring (resp. universal deformation) of $\rhobar$ relatively unramified outside $S$ and $(R_S^\univ)^0$ (resp. $(\rho_S^{\univ})^0$) to be its constant determinant counterpart (see \S\ref{sec:nt}).
We will now define minimal deformation problems, following~\cite{CG}.
\begin{defi}\label{mindef}
Let $R$ be an object of $\calC$ and let $\rhobar : G_K \to \GL_2(\FF)$ be an absolutely irreducible representation satisfying assumption~\eqref{ass:5} of Set-up~\ref{setup:modularity}. A deformation $\rho : G_K \to \GL_2(R)$ of $\rhobar$ is called minimal if it satisfies all the following properties:
\begin{enumerate}
\item $\det\rho = \widehat{\det(\rhobar)}$.
\item $\rho$ is unramified at primes at which $\rhobar$ is unramified.
\item If $\ell$ is a vexing prime, then $\rho(I_{\ell}) \simeq \rhobar(I_{\ell})$.
\item If $\ell$ is a prime such that $\rhobar$ is ramified at $\ell$ and $\rhobar|_{G_{K_{\ell}}}$ is not absolutely irreducible, then $\rho^{I_{\ell}}$ is a rank $1$ direct summand of $\rho$ as an $R$-module.
\end{enumerate}
For $K=\QQ$ and odd $\rhobar$, this is just \cite[Definition $3.1$]{CG}.
\end{defi}

It follows, from the proof of \cite[Theorem $2.41$]{DDT}, that the functor from $\calC$ to the category of sets sending an object $R$ of $\calC$ to the set of continuous, minimal deformations of $\rhobar$ to $\GL_2(R)$ is representable by a ring in~$\calC$ (see \cite[Section $3.1$]{CG} as well).
We will denote this ring by $R^\minim$ and we will denote the universal minimal deformation by $\rho^\minim$.

\begin{lem}\label{lem:min-univ}
Let $\rhobar : G_K \to \GL_2(\FF)$ be an absolutely irreducible representation satisfying assumptions~\eqref{ass:4} and~\eqref{ass:5} of Set-up~\ref{setup:modularity}. 
Let $D$ be the tame Artin conductor of $\rhobar$ and let $S$ be the union of $S_\infty$ and the set of primes $\ell$ of $K$ such that $\ell \mid D$, $\rhobar|_{G_{K_{\ell}}}$ is absolutely irreducible and $\ell$ is not a vexing prime. 
Then $R^\minim \simeq (R^{\univ}_S)^0$.
\end{lem}

\begin{proof}
A minimal deformation is unramified at primes of $K$ not dividing~$D$.
Let $\ell$ be a prime dividing~$D$. By the definition of minimal deformations, if $\ell$ is a vexing prime, then $\rho^\minim$ is a deformation of~$\rhobar$ which is relatively unramified at $\ell$, i.e.\ if $\ell$ is a vexing prime, then $\rho^\minim(I_{\ell}) \simeq \rhobar(I_{\ell})$.
If $\rhobar|_{G_{K_{\ell}}}$ is not absolutely irreducible, then we have assumed that the subspace $(\rhobar)^{I_{\ell}}$ has dimension~$1$. So, $\rhobar|_{I_{\ell}} = 1 \oplus \delta$ for some non-trivial character~$\delta$. The minimality condition means that $(\rho^\minim)^{I_{\ell}}$ is a free $R^\minim$-module of rank~$1$ which is a direct summand of $\rho^\minim$ as an $R^\minim$-module. As $\det\rho^\minim$ is the Teichm\"uller lift of $\det\rhobar$, we get that $\rho^\minim|_{I_{\ell}} \simeq \begin{pmatrix} 1 & *\\ 0 & \widehat{\delta}\end{pmatrix}$. We have two cases:
\begin{enumerate}
\item \textbf{$\delta$ is tamely ramified:} In this case, $\rho^\minim(I_{\ell})$ factors through the tame inertia quotient of~$I_{\ell}$ and is hence abelian. This means that the $*$ above is necessarily $0$ as $\widehat{\delta}$ is non-trivial. Therefore, we get that $\rho^\minim|_{I_{\ell}} \simeq 1 \oplus \widehat{\delta}$. Thus, $\rho^\minim$ is a deformation of~$\rhobar$ which is relatively unramified at $\ell$.
\item \textbf{$\delta$ is wildly ramified:} Let $W_{\ell}$ be the wild inertia group at $\ell$. As $\ell \nmid p$, $W_{\ell}$ does not admit any non-trivial pro-$p$ quotient. So, $\rho^\minim|_{W_{\ell}} \simeq 1 \oplus \widehat{\delta}|_{W_{\ell}}$. As $W_{\ell}$ is a normal subgroup of $I_{\ell}$ and $1 \neq \widehat{\delta}|_{W_{\ell}}$, we see that the submodules of $\rho^\minim$ on which $W_{\ell}$ acts via $1$ or $\widehat{\delta}$ are also $I_{\ell}$-invariant. Therefore, we get that $\rho^\minim|_{I_{\ell}} \simeq 1 \oplus \widehat{\delta}$. Hence, $\rho^\minim$ is a deformation of~$\rhobar$ which is relatively unramified at $\ell$.
\end{enumerate}
Note that the primes considered above are exactly the primes of $K$ which divide $D$ but are not in $S$. Being a minimal deformation does not put any conditions on any other primes of $K$ dividing~$D$. Thus, $\rho^\minim$ is relatively unramified outside~$S$ and has constant determinant. On the other hand, any deformation of $\rhobar$ which is relatively unramified outside~$S$ with constant determinant is also minimal by definition. Hence, we get morphisms $\alpha : R^\minim \to(R^{\univ}_S)^0$ and $\beta : (R^{\univ}_S)^0 \to R^\minim$. It follows, from looking at the corresponding deformations, that both morphisms $\alpha \circ \beta$ and $\beta \circ \alpha$ are the identity and hence, $R^\minim \simeq (R^{\univ}_S)^0$.
\end{proof}

\begin{prop}\label{prop:min-univ}
Let $\rhobar : G_K \to \GL_2(\FF)$ be a dihedral representation satisfying all the assumptions of Set-up~\ref{setup:modularity}.
Then $\rho^\minim$ is dihedral.
\end{prop}

\begin{proof}
From Lemma~\ref{lem:min-univ}, we know that $R^\minim \simeq (R^{\univ}_S)^0$, where $S$ is the union of $S_\infty$ and the set of primes $\ell$ of $K$ such that $\ell \mid D$, $\rhobar|_{G_{K_{\ell}}}$ is absolutely irreducible and $\ell$ is not a vexing prime.
Moreover, the proof of Lemma~\ref{lem:min-univ} implies that this isomorphism takes $\rho^\minim$ to $(\rho^\univ_S)^0$.
Let $S_1$, $S_2$, $S_3$ and $S_0$ be the set of primes of $K$ defined in \S\ref{sec:nt} right after Proposition~\ref{prop:CL-Ind}. We keep using Notation~\ref{not:nt} and prove that $S \cap S_0 = \emptyset$. Indeed, let $\ell \in S \cap S_0$ be a finite place.
Firstly, $\ell \in S$ implies that $\rhobar|_{G_{K_{\ell}}}$ is absolutely irreducible and $\ell$ is not a vexing prime. As $\rhobar|_{G_{K_{\ell}}}$ is absolutely irreducible and $\rhobar$ ramifies at $\ell$, the assumption that $\ell$ is not a vexing prime means that either $[K_{\ell}(\mu_p):K_{\ell}] \neq 2$ or $\rhobar|_{I_{\ell}}$ is absolutely irreducible. Now, if $q \in S_1 \cup S_2$, then it follows, from the definitions of $S_1$ and $S_2$, that the projective image of $\rhobar|_{G_{K_q}}$ is cyclic. Therefore, the image of $\rhobar|_{G_{K_q}}$ is abelian and, hence, $\rhobar|_{G_{K_q}}$ is not absolutely irreducible. So, $\ell \not\in S_1 \cup S_2$ which means that $\ell \in S_3$. From the definition of $S_3$, we get that $[K_{\ell}(\mu_p):K_{\ell}]=2$ and $|\ad\rhobar(I_{\ell})|=2$. Thus, the projective image of $\rhobar|_{I_{\ell}}$ is a cyclic group (of order~$2$). Hence, the image of $\rhobar|_{I_{\ell}}$ is abelian. So, it follows that $\rhobar|_{I_{\ell}}$ is not absolutely irreducible. This contradicts our assumption that $\ell$ is not a vexing prime. Hence, we get that $S \cap S_0 =\emptyset$.
This allows us to apply Theorem~\ref{thm:nt-dih} to conclude that $\rho_S^\univ$ is dihedral. The result follows. 
\end{proof}

\begin{proof}[Proof of Theorem~\ref{minthm}.]
Let $S$ be as in Proposition~\ref{prop:min-univ}, so that we can describe $R^\minim$ as $(R_S^\univ)^0$. By Proposition~\ref{prop:min-univ} and Proposition~\ref{prop:Rmodular}, there exists a quotient $\TT_\dih$ of $\TT_{1,D'}$, a deformation $\rho_\dih : G_K \to \GL_2(\TT_\dih)$ of $\rhobar$ and a map $\phi : \TT_\dih \to R^\minim$ such that $\phi \circ \rho_\dih = \rho^{\minim}$. Here $D'$ is an ideal of $\calO_K$ such that $D | D'$ and $D'/D$ is only divisible by primes lying in $S$.
By construction the determinant of $\rho_\dih$ is equal to the Teichm\"uller lift of~$\rhobar$.
By Proposition~\ref{prop:min-univ}, $\rho^\minim$ factors through $\Gal(M^{\rhobar}(S)/K)$.
Hence, from the proof of Proposition~\ref{prop:Rmodular}, it follows that $\rho_\dih$ factors through $\Gal(M^{\rhobar}(S)/K)$.
Hence $\rho_\dih$ a deformation of $\rhobar$ that is relatively unramified outside~$S$ with constant determinant.
In view of Proposition~\ref{prop:min-univ}, $\rho_\dih$ is hence a minimal deformation of~$\rhobar$.
Consequently, Theorem~\ref{thm:ReqT} finishes the proof. 
\end{proof}

\begin{rem}\label{rem:non-lift}
If $\Hom_{\FF_p[\Gbar^\ad]}(A(M^{\ad})/p A(M^{\ad}), I(\chibar/\chibar^{\sigma})) \neq 0$, then, as seen in Remark~\ref{rem:nt-dih}, there exists a non-dihedral infinitesimal deformation which is relatively unramified everywhere.
Hence, we get a non-dihedral infinitesimal minimal deformation  which means that, in this case, the universal minimal deformation is not dihedral. So, the methods of this article will not be useful to prove an $R^\minim = \TT$ theorem.
\end{rem}

Note that we can remove assumption~\ref{ass:5} from Set-up~\ref{setup:modularity} and look at deformations unramified outside $S (\supseteq S_{\infty})$ with constant determinant for a finite set $S$ of primes of $K$ with $S \cap (S_0 \cup S_p) = \emptyset$ instead of minimal deformations. In this case, our methods will not give an $R=\TT$ theorem, but we can still conclude the following:

\begin{prop}
Let $\rhobar$ be a dihedral representation satisfying all the assumptions of Set-up~\ref{setup:modularity} except possibly assumption~\ref{ass:5}. Let $\mathcal{R}$ be the ring of integers of a finite extension of $\QQ_p$ such that the residue field of $\mathcal{R}$ contains $\FF$. Let $S (\supset S_{\infty})$ be a finite set of primes of $K$ with $S \cap (S_0 \cup S_p) = \emptyset$. If $\rho : G_K \to \GL_2(\mathcal{R})$ is a deformation of $\rhobar$ with constant determinant which is relatively unramified outside $S$, then there exists a classical Hilbert modular eigenform $f$ of parallel weight one over $K$ such that $\rho$ is isomorphic to the Galois representation attached to $f$.
\end{prop}

\begin{proof}
From Theorem~\ref{thm:nt-dih}, it follows that $\rho$ is a dihedral representation. From Lemma~\ref{modlem}, we conclude the existence of the parallel weight one eigenform $f$ over $K$ having the required property. 
\end{proof}

See \cite[Theorem 1.1]{C} for a similar but much stronger result for $K=\QQ$.

\begin{rem}
To conclude that $\rho$ is dihedral we do not need the hypotheses~\ref{ass:1} ($K$ is totally real) and $\ref{ass:2}$ ($\rhobar$ is totally odd). Hence, if we further remove the assumptions that $K$ is totally real and $\rhobar$ is totally odd from the Proposition above, then we can still conclude, by automorphic induction (\cite{G}, \cite{R}), that $\rho$ comes from an automorphic representation for $\GL_2(K)$.
\end{rem}

\section{Examples}\label{sec:examples}

In this section, we present several examples of irreducible dihedral representations~$\rhobar: G \to \GL_2(\FF)$ as in the rest of the article and determine whether their universal deformation relatively unramified outside a finite set~$S$ is dihedral or not.
Most of the time we take $S=S_\infty$, {\it i.e.} we consider deformations that are relatively unramified at all finite places.

\subsection{Examples of $\rhobar$ such that $\rho_{S_{\infty}}^\univ$ is non-dihedral for $p=2$}
For $p=2$, there is, in a sense, a generic source of examples where a dihedral representation deforms infinitesimally into a non-dihedral one.
Denote by $S_n$ the symmetric group on $n$ letters.
We start with an $S_4$-extension $M/K$ of number fields.
We know that the double-transpositions generate the normal subgroup $V_4 = \ZZ/2\ZZ \times \ZZ/2\ZZ$ of~$S_4$ the quotient of which is isomorphic
to~$S_3 \cong D_3 \cong \SL_2(\FF_2)$. Moreover, the surjection $S_4 \twoheadrightarrow S_3$ is split by the natural map $S_3 \hookrightarrow S_4$ and the conjugation action by $S_3$ on $V_4$ is non-trivial and thus, after identifying $S_3 \cong \SL_2(\FF_2)$, we have that $V_4$ becomes $I=\Ind_H^G(\chibar) = I(\chibar/\chibar^\sigma)$.
We thus find the following commutative diagram with split exact rows:
$$ \xymatrix@=.5cm{
 & & \Gal(M/K) \ar@{->}[d]^\sim  \ar@{->>}[rd] \ar@{->}@/_2pc/[dd]_(.7){\rho} \ar@{->}[rdd]_(.6){\rhobar} \\
0 \ar@{->}[r]& V_4 \ar@{->}[r] \ar@{^(->}[d]^{\id \oplus 0} & S_4 \ar@{->}[r] \ar@{^(->}[d] & S_3 \ar@{->}[r]\ar@{=}[d] & 0 \\
0 \ar@{->}[r]& I \oplus C(1) \ar@{->}[r] & \SL_2(\FF_2[\epsilon]/(\epsilon^2))\ar@{->}[r] & \SL_2(\FF_2) \ar@{->}[r] & 0 \\
}$$
We thus see that $\rho$ is an infinitesimal deformation of the representation~$\rhobar$.
In order to satisfy Definition~\ref{defi:dihedral}, we must extend the scalars of both $\rhobar$ and $\rho$ from $\FF_2$ to~$\FF_4$. Then $\rhobar$ is a dihedral
representation admitting the non-dihedral deformation~$\rho$.
This situation occurs, for instance, for $K=\QQ$ and the $S_3$-extension of~$\QQ$ given by the Hilbert class field of $Q(\sqrt{229})$.
This is a totally real field and its ray class field ramifying only at infinity provides the desired $S_4$-extension of~$\QQ$.

\subsection{Strategy for finding examples of $\rhobar$ such that $\rho_{S_{\infty}}^\univ$ is dihedral for $p>2$}
For the other examples, we take $S=S_\infty$ (i.e.\ we only consider relatively unramified deformations), $p>2$ and $\rhobar$ that are unramified above~$p$.
In that case, the only condition in Theorem~\ref{thm:nt-dih} is that the induced representation $I(\chibar/\chibar^\sigma)$ does not occur
in the $p$-part of the class group of $M^\ad$.

Let $K$ be a number field and $L$ a quadratic extension of~$K$.
For simplicity, we shall only consider cases when a chosen odd prime $q \neq p$ exactly divides the class number of~$L$. Let $M/L$ be the corresponding cyclic extension with Galois group $\ZZ/q\ZZ$ inside the Hilbert class field of~$L$. Note that $M/K$ is Galois.
We shall further assume that the Galois group of $M/K$ is not $\ZZ/2\ZZ\times \ZZ/q\ZZ$, whence it automatically is isomorphic to the dihedral group~$D_q$.
We fix a character $\chibar: G_L \to \FF_{p^r}^\times$ of kernel~$G_M$ with $r$ the multiplicative order of~$q$ modulo~$p$.
Note that $\chibar^\sigma = \chibar^{-1}$ for any $\sigma \in G_K \setminus G_L$.

Next consider the maximal elementary abelian $p$-extension $M_1/M$ (resp.\ $L_1/L$) inside the Hilbert class field of~$M$ (resp.\ of~$L$).
We shall consider the group $\calG := \Gal(M_1/M)$ as $\FF_p[\Gal(M/K)]$-module.
Then two mutually exclusive cases can arise.

\begin{enumerate}[(1)]
\item \label{case:1} $[L_1:L]=[M_1:M]$.

This happens if and only if $M_1 = L_1M$.
This condition is furthermore equivalent to the only simple $\FF_p[\Gal(M/K)]$-modules occuring in $\calG$ being $1$-dimensional (as $\FF_p$-vector space) and, thus, either $C(1)$ or $C(\epsilon)$ with $\epsilon: G_K \twoheadrightarrow \Gal(L/K) \cong \{\pm 1\}$.

In this case, $\calG$ is trivial as an $\FF_p[\Gal(M/L)]$-module. Let $\rhobar := \Ind_{G_L}^{G_K}(\chibar^b)$ for any $b \in \FF_q^\times$.
In the notation of Set-up~\ref{setup:two}, $\Hbar^\ad = \Gal(M/L)$ and we have that $\calG$ is trivial as $\FF_p[\Hbar^\ad]$-module.
For any infinitesimal deformation $\rho: G_K \to \GL_2(R)$ of~$\rhobar$ that is everywhere relatively unramified, the corresponding $\Gamma_\rho$ is a quotient of~$\calG$,
and thus trivial as $\FF_p[\Hbar^\ad]$-module.
Consequently, by Theorem~\ref{thm:inf-dih}, $\rho$ is dihedral, so that by Theorem~\ref{thm:inf-dih-univ}, the universal relatively unramified deformation $\rho^\univ$ of $\rhobar$ is dihedral.

\item \label{case:2} $[L_1:L]<[M_1:M]$.

This is the case if and only if $\calG$ contains an irreducible $\FF_p[\Gal(M/K)]$-module of $\FF_p$-dimension at least~$2$.

In this case, by the representation theory of the dihedral group $D_q$, this representation is then $I := \Ind_{G_L}^{G_K}(\chibar^a)$ (defined over its minimal
field of definition, but viewed as $\FF_p[\Gal(M/K)]$-module) for some $a \in \FF_q^\times$. Let now $b \in \FF_q^\times$ be such that $2b=a$.
Now consider $\rhobar := \Ind_{G_L}^{G_K}(\chibar^b)$. In the notation of Set-up~\ref{setup:two}, $\Gbar^\ad = \Gal(M/K)$.
Then $I = I((\chibar^b)/(\chibar^b)^\sigma)$, which occurs in $\ad(\rhobar)$ as $\FF_p[\Gbar^\ad]$-module according to Lemma~\ref{lem:indec}.
By Proposition~\ref{prop:infi-lift}, there is thus a deformation $\rho_I$ of~$\rhobar$, which is non-dihedral according to Theorem~\ref{thm:inf-dih}.
Consequently, the universal everywhere relatively unramified deformation $\rho^\univ$ of $\rhobar$ is non-dihedral.
\end{enumerate}

Note that we are sure to be in case~\eqref{case:1} if $p$ does not divide the quotient of the class number of $M$ by the class number of~$L$.
Conversely, suppose that the $p$-part of the Hilbert class field of~$L$ equals~$L_1$ (i.e.\ the $p$-part of the class group is an
elementary abelian $p$-group) and assume $[L_1:L]=[M_1:M]$.
Let $M_2$ be the $p$-part of the Hilbert class field of~$M$.
Then the $p$-Frattini quotient of $\Gal(M_2/M)$ is its maximal elementary abelian quotient $\Gal(M_1/M)$.
The group $\Gal(M/L)$ is of order prime-to~$p$ and acts trivially on $\Gal(M_1/M)$,
and, thus, by Proposition~\ref{prop:coprime-aut}, it also acts trivially on $\Gal(M_2/M)$. There is thus an extension $L_2/L$ such that
$M_2=L_2M$ with $L_2$ inside the Hilbert class field of~$L$. By assumption, $L_2=L_1$ and consequently $M_2=M_1$.
This means that $p$ does not divide the quotient of the class number of~$M$ by the class number of~$L$.

We summarise the discussion so far: Starting with a number field~$K$, we take a quadratic extension $L/K$ such that an odd prime $q \neq p$ exactly
divides the class number of~$L$ and we let $M$ be the corresponding cyclic $\ZZ/q\ZZ$-extension of~$L$ inside the Hilbert class field of~$L$,
assuming that $\Gal(M/K)$ is dihedral and that the $p$-part of the class group of~$L$ is of exponent~$p$.
Then we are in case~\eqref{case:1} if and only if $p$ does not divide the quotient of the class number of~$M$ by the class number of~$L$;
otherwise we are in case~\eqref{case:2}.

This observation allows us to derive concrete examples of dihedral $\rhobar$ satisfying desired hypotheses by computing class numbers of abelian extensions.
To be precise, after computing the class numbers, we got examples of both case~\eqref{case:1} and case~\eqref{case:2} (see \S\ref{sec:eg1}, \S\ref{sec:eg2}, \S\ref{sec:eg3} and \S\ref{sec:eg4}).
Thus we get examples of dihderal $\rhobar$'s of both kinds, one where everywhere relatively unramified universal deformation of $\rhobar$ is dihedral (case~\eqref{case:1}) and one where it is not dihedral (case~\eqref{case:2}).
All computations were performed using Magma~\cite{magma} under the assumption of the Generalised Riemann Hypothesis (GRH).

\subsection{Examples of cases~\eqref{case:1} and \eqref{case:2} for $K=\QQ$}\label{sec:eg1}

The first set of examples is for the base field~$K=\QQ$ and aims at providing examples for both cases. We want these examples
to be non-trivial, in the sense that there does exist a non-trivial dihedral infinitesimal deformation.
We did some small systematic calculation among imaginary quadratic fields~$L$ of class numbers $15, 21, 33, 35$. The four numbers are
products of two distinct primes and we took $p$ and $q$ to be either choice.
The results are summarised in the following table.
\begin{longtable}{l|l||p{6.5cm}|p{6.5cm}}
$p$ & $q$ & fields $L$ & results \\
\hline
\hline
$5$ & $3$ & all imaginary quadratic fields of class number $15$ & 
\begin{minipage}{7cm} case~\eqref{case:2}: discriminants: $-4219$, $-19867$ \\ case~\eqref{case:1}: all 66 others \end{minipage}\\
\hline
$3$ & $5$ & all imaginary quadratic fields of class number $15$ of abs.\ value of discriminants $\le 19387$ &  case~\eqref{case:1}: all 60 fields \\
\hline
$7$ & $3$ & all imaginary quadratic fields of class number $21$ of abs.\ value of discriminants $\le 14419$ & 
\begin{minipage}{7cm} case~\eqref{case:2}: discriminant: $-8059$ \\ case~\eqref{case:1}: all 41 others \end{minipage}\\
\hline
$3$ & $7$ & all imaginary quadratic fields of class number $21$ of abs.\ value of discriminants $\le 5867$ &  case~\eqref{case:1}: all 18 fields \\
\hline
$11$ & $3$ & all imaginary quadratic fields of class number $33$ of abs.\ value of discriminants $\le 28163$ &  case~\eqref{case:1}: all 38 fields \\
\hline
$3$ & $11$ & all imaginary quadratic fields of class number $33$ of abs.\ value of discriminants $\le 1583$ &  case~\eqref{case:1}: all 2 fields \\
\hline
$7$ & $5$ & all imaginary quadratic fields of class number $35$ of abs.\ value of discriminants $\le 16451$ &  case~\eqref{case:1}: all 25 fields \\
\hline
$5$ & $7$ & all imaginary quadratic fields of class number $35$ of abs.\ value of discriminants $\le 4931$ &  case~\eqref{case:1}: all 7 fields \\
\end{longtable}

\subsection{Examples of case~\eqref{case:1} for quadratic extensions of $\QQ$}\label{sec:eg2}
We also looked at examples for quadratic base fields~$K$. In the first set of examples of this kind, let $K=\QQ(\sqrt{d})$ for $d=2,5,13,17$.
We ran through some CM extensions $L$ of~$K$ that admit a class number that is divisible by two odd primes $p,q$ to the first power,
with $q$ being $3$ or~$5$. In total we computed $103$ fields with these properties. All of them fell into case~\eqref{case:1}.
Note that this also gives examples when our $R^\minim=\TT$-result (Theorem~\ref{minthm}) holds because in these cases $\rhobar$ is unramified above~$p$,
totally odd and the condition on the inertia invariants is satisfied because the orders of the inertia groups are $1$ or~$2$.

In order not to only treat real quadratic fields, we also searched for and found a case-\eqref{case:1} example for
the imaginary quadratic field $K=\QQ(i)$ with $i=\sqrt{-1}$ for $q=3$ and $p=5$. It is obtained for the quadratic extension $L=K(\sqrt{-79i + 84})$
of~$K$, which has class number~$30$. The field $M$ is the unique unramified degree~$3$ extension of~$L$, and its class number is~$10$.

\subsection{Examples of case~\eqref{case:2} for $K=\QQ$}\label{sec:eg3}
Since in the range where we looked, case~\eqref{case:2} seems to be rather rare in the above set-up, we looked explicitly for
case~\eqref{case:2}-examples for the base field $K=\QQ$. We ran through imaginary quadratic fields of negative prime discriminant
(for each line, up to the largest value appearing in the line).
The results are summarised in the following table.
\begin{longtable}{l|l||p{13cm}}
$q$ & $p$ & negative prime discriminants with case~\eqref{case:2}\\
\hline
\hline
$3$ & $5$ & 673, 1193, 1993, 1999, 2819, 4219, 4637, 5087, 5437, 5791, 5897, 7907, 8803, 9013, 9103, 9349, 9551, 9857, 10391, 10453, 10937, 11491, 13873\\
\hline
$3$ & $7$ & 2749, 4513, 5717, 6581, 8059, 9613, 9733, 11971\\
\hline
$3$ & $11$ & 3061\\
\hline
$3$ & $13$ & 9397\\
\hline
$5$ & $11$ & 709, 1489\\
\hline
$5$ & $19$ & 3389, 3701\\
\hline
$7$ & $13$ & 997\\
\end{longtable}

\subsection{Example of case~\eqref{case:2} for real quadratic fields}\label{sec:eg4}
We also looked for a case-\eqref{case:2} example over a real quadratic field.
We found one for $K=\QQ(\sqrt{13})$, $q=3$, $p=5$. Let $\omega = \frac{1+\sqrt{13}}{2}$ and $\alpha = 15\omega -73$ and set
$L = K(\sqrt{\alpha})$. The norm of $\alpha$ is the prime $3559$.
The class number of $L$ equals~$24 = 2^3 \cdot 3$ and we let $M$ be the unique unramified cyclic extension of~$L$ of degree~$3$.
The class number of $M$ equals $200 = 2^3 \cdot 5^2$, so that the quotient of the two class numbers is $5^2$ and we are indeed
in case~\eqref{case:2} by the above criterion.

\bibliography{References}
\bibliographystyle{alpha}

\bigskip
\noindent Shaunak V.\ Deo\\
School of Mathematics\\
 Tata Institute of Fundamental Research\\
 Homi Bhabha Road\\
 Mumbai 400005\\
 India\\
\url{deoshaunak@gmail.com}

\bigskip

\noindent Gabor Wiese\\
University of Luxembourg\\
Department of Mathematics\\
Maison du nombre\\
6, avenue de la Fonte\\
L-4364 Esch-sur-Alzette\\
Grand-Duchy of Luxembourg\\
\url{gabor.wiese@uni.lu}
\end{document}